\documentclass[runningheads, envcountsame]{llncs}
\usepackage{graphicx}
\graphicspath{{figs/}}

\usepackage[hidelinks]{hyperref}
\usepackage[dvipsnames]{xcolor}

\usepackage{cite}
\usepackage{microtype}

\usepackage{amsmath,amssymb,amsthm,thmtools} 
\usepackage[shortlabels]{enumitem}
\usepackage[capitalize]{cleveref}
\crefname{cond}{Condition}{Conditions}

\DeclareMathOperator{\qn}{qn}
\DeclareMathOperator{\uqn}{qn_u}
\DeclareMathOperator{\lqn}{qn_\ell}
\DeclareMathOperator{\pn}{pn}
\DeclareMathOperator{\upn}{pn_u}
\DeclareMathOperator{\lpn}{pn_\ell}
\DeclareMathOperator{\len}{len}
\DeclareMathOperator{\fwdP}{fwd_P}
\DeclareMathOperator{\mad}{mad}
\newcommand{\calQ}{\ensuremath{\mathcal{Q}}}
\newcommand{\calP}{\ensuremath{\mathcal{P}}}
\newcommand{\calA}{\ensuremath{\mathcal{A}}}
\newcommand{\calB}{\ensuremath{\mathcal{B}}}
\newcommand{\calC}{\ensuremath{\mathcal{C}}}
\newcommand{\calF}{\ensuremath{\mathcal{F}}}
\renewcommand{\leq}{\leqslant}
\renewcommand{\geq}{\geqslant}

\begin{document}
\title{Linear Layouts of Complete Graphs} 
\author{%
    Stefan Felsner\inst{1}\thanks{Partially supported by DFG grant FE 340/13-1} \and
    Laura Merker\inst{2} \and
    Torsten Ueckerdt\inst{2} \and
    Pavel Valtr\inst{3}
}
\authorrunning{S. Felsner et al.}
\institute{%
    Institute of Mathematics, Technische Universität Berlin, Germany
    \email{felsner@math.tu-berlin.de}
    \and
    Institute of Theoretical Informatics, Karlsruhe Institute of Technology, Germany 
    \email{laura.merker2@kit.edu, torsten.ueckerdt@kit.edu}
    \and
    Department of Applied Mathematics, Faculty of Mathematics and Physics, Charles University, Prague, Czech Republic
    \email{valtr@kam.mff.cuni.cz}
}
\maketitle
\begin{abstract}
    A \emph{page} (\emph{queue}) with respect to a vertex ordering of a graph is a set of edges such that no two edges cross (nest), i.e., have their endpoints ordered in an \textsc{abab}-pattern (\textsc{abba}-pattern).
    A \emph{union page} (\emph{union queue}) is a vertex-disjoint union of pages (queues).
    The \emph{union page number} (\emph{union queue number}) of a graph is the smallest $ k $ such that there is a vertex ordering and a partition of the edges into $ k $ union pages (union queues).
    The \emph{local page number} (\emph{local queue number}) is the smallest $ k $ for which there is a vertex ordering and a partition of the edges into pages (queues) such that each vertex has incident edges in at most $ k $ pages (queues).
    
    We present upper and lower bounds on these four parameters for the complete graph $K_n$ on $n$ vertices. In three cases we obtain the exact result up to an additive constant. In particular, the local page number of $K_n$ is $n/3 \pm \mathcal{O}(1) $, while its local and union queue number is $(1-1/\sqrt{2})n \pm \mathcal{O}(1) $. The union page number of $ K_n $ is between $ n/3 - \mathcal{O}(1) $ and $ 4n/9 + \mathcal{O}(1) $.

\end{abstract}

\section{Introduction}
\label{sec:introduction}

A \emph{linear layout} of a graph consists of a vertex ordering together with a partition of the edges.
For a fixed vertex ordering $ \prec $, we say that two independent edges $ vw $ and $ xy $ with $ v \prec w $ and $ x \prec y $ 
\begin{itemize}
 \item \emph{nest} if $ v \prec x \prec y \prec w $ or $ x \prec v \prec w \prec y $ and
 \item \emph{cross} if $ v \prec x \prec w \prec y $ or $ x \prec v \prec y \prec w $.
\end{itemize}
A \emph{queue}, respectively a \emph{page} (also called \emph{stack}), is a subset of edges that are pairwise non-nesting, respectively non-crossing.
A \emph{queue layout}, respectively a \emph{book embedding}, is a linear layout whose partition of the edge set consists of queues, respectively pages.
The \emph{queue number} and \emph{page number} (also called \emph{stack number} or \emph{book thickness}) denote the smallest $ k $ such that there is a linear layout consisting of at most $ k $ queues, respectively at most $k$ pages.
Both queue layouts and book embeddings were intensively investigated in the past decades, where complete graphs are one of the very first considered graph classes~\cite{bk-79,hr-92}.

Queue layouts and book embeddings model how the edges of a graph can be assigned to and processed by queues, respectively stacks.
It was first asked by Heath, Leighton, and Rosenberg~\cite{hlr-92} whether queues or stacks are more powerful in this context.
Recently, Dujmović et al.~\cite{dehmw-20} partly answered this question by presenting a class of graphs with bounded queue number that needs an unbounded number of stacks, showing that stacks are not more powerful than queues for representing graphs.
In the classical variant in this question, the total number of necessary queues or stacks serves as a measure for the power of queues and stacks.
Local and union variants relax the setting by allowing more queues or stacks, as long as each vertex is touched by a small number of queues, respectively stacks, or if they operate on vertex-disjoint subgraphs.

Local and union variants of queue layouts and book embeddings were recently introduced by the second and third author~\cite{mu-19,mu-20}.
A linear layout is called \emph{$ k $-local} if each vertex has incident edges in at most $ k $ parts of the edge partition.
The \emph{local queue number}, respectively \emph{local page number}, is the smallest $ k $ such that there is a $ k $-local queue layout, respectively a $ k $-local book embedding for $ G $.
Given a fixed vertex ordering, a \emph{union queue} (\emph{union page}) is a vertex-disjoint union of queues (pages).
The \emph{union queue number}, respectively \emph{union page number}, then is the smallest $ k $ such that there is a linear layout whose edge partition consists of at most $ k $ union queues, respectively union pages.

Local and union variants have been considered for numerous graph decomposition parameter, such as boxicity, interval numbers, planar thickness, poset dimension, several arboricities, and many more.
Especially in recent years, there has been a lot of interest in these variants in various directions, see e.g.~\cite{bsu-18,el-20,mm-21,kmmssuw-20,dfgklnu-21,ku-16}.

In this paper, we continue the investigation of local and union variants of linear layouts, that is, queue numbers and page numbers.
We establish bounds on the respective graph parameters for complete graphs, which interestingly turns out to be non-trivial; in contrast to their classic counterparts.
Our results show that the local and union variants of queue numbers and page numbers for $n$-vertex complete graphs are located strictly between the trivial lower bound of $ (n - 1)/4 $ due to the density~\cite[see also \cref{sec:prel}]{mu-19,mu-20} and $ \lfloor n/2 \rfloor $, respectively $ \lceil n/2 \rceil $, which is the queue number, respectively page number, of complete graphs~\cite{bk-79,hr-92}.

\subsubsection*{Outline.}
In \cref{sec:prel}, we survey the relation between local and union variants of queue layouts and book embeddings.
We give a lower bound on the local queue number of complete graphs and show how to obtain the same upper bound for the union queue number in \cref{sec:queue-number} (up to an additive constant).
\cref{sec:page-number} continues with a lower and a matching (up to an additive constant) upper bound on the local page number.
We also present an upper bound on the union page number and discuss properties that an improved bound needs to satisfy in this section.

\subsubsection*{Our results.} 
Both the local queue number and the union queue number of $K_n$ are linear in $n$ with the leading
coefficient being $1-1/\sqrt{2} \approx 0.29289$. For the local queue number, the error term $ \mathcal{O}(1) $ is small; it is between $-0.21$ and $+2$.
\begin{theorem}\label{thm:queue-number} 
 The local queue number and the union queue number of $K_n$ satisfy
 \begin{equation*}
  \lqn(K_n) = (1-\frac{1}{\sqrt{2}})n \pm \mathcal{O}(1) \text{ and } 
  \uqn(K_n) = (1-\frac{1}{\sqrt{2}})n \pm \mathcal{O}(1).
 \end{equation*}
\end{theorem}

The local page number of $K_n$ is also linear in $n$ with the leading
coefficient being $1/3$. In this case, the error term $ \mathcal{O}(1) $
is between $0$ and $4$.
The local page number also gives a lower bound on the union page number, whereas the leading coefficient of our upper bound is $ 4/9 $.
\begin{theorem}\label{thm:page-number} 
 The local page number and the union page number of $K_n$ satisfy
 \begin{equation*}
  \frac{1}{3}n \pm \mathcal{O}(1) = \lpn(K_n) \leq
  \upn(K_n) \leq \frac{4}{9}n + \mathcal{O}(1).
 \end{equation*}
\end{theorem}

\section{Preliminaries}
\label{sec:prel}

In this section, we summarize some known results on local and union linear layouts, which are presented in~\cite{mu-19,mu-20} in detail.

First, we have $ \lpn(G) \leq \upn(G) \leq \pn(G) $ and $ \lqn(G) \leq \uqn(G) \leq \qn(G) $, where the gap between the union and global variants can be arbitrarily large.
There are graph classes (e.g.\ $ k $-regular graphs for $ k \geq 3 $) with bounded union queue number and union page number but unbounded queue number and page number.
In contrast, the local page number, the local queue number, the union page number, and the union queue number are all tied to the maximum average degree, which is defined by $ \mad(G) = \max \{ 2 \, |E(H)| / |V(H)| \colon H \subseteq G, H \neq \emptyset \} $.
In particular, we have the following connection between the maximum average degree and the union queue number, respectively the union page number.
We follow the proof of a similar statement for book embeddings in~\cite{mu-19}.
However, we stress that the vertex ordering is arbitrary as we make use of this in the proof of \cref{lem:uqn-UB-triangle}.
\begin{proposition}
    \label{prop:uqn-mad}
    Every graph $ G $ 
    admits a $ (\mad(G) + 2) $-union queue layout and a $ (\mad(G) + 2) $-union book embedding with any vertex ordering.
\end{proposition}

\begin{proof}
    Nash-Williams~\cite{n-64} proved that every graph can be partitioned into at most $ \mad(G) / 2 + 1 $ forests.
    Each forest, in turn, can be partitioned into two star forests~\cite{aa-89}.
    Choosing an arbitrary vertex ordering, each of the $ \mad(G) + 2 $ star forests is both a union queue and a union page as the edges of a star can neither nest nor cross.
\end{proof}    

In addition, the local queue number and the local page number are lower-bounded by $ \mad(G)/4 $, which gives a lower bound of $ (n - 1)/4 $ for $ K_n $.  
    
In the following sections, we consider the complete graph $K_n$ with vertex set
$V(K_n) = \{v_1,\ldots,v_n\}$.  Due to symmetry, we may throughout assume that
the vertex ordering of $K_n$ is given by $v_1 \prec \cdots \prec v_n$.
Partitioning the edges is the difficult part.

\section{Local and Union Queue Numbers}
\label{sec:queue-number}

We first establish the lower bound of \cref{thm:queue-number}.
As the union queue number is lower-bounded by the local queue number, we only consider the latter.
Note that $ (9-4\sqrt{2})/16 \approx 0.21 $.
\begin{lemma}\label{lem:queue-LB} For any $n$ we have
  $\displaystyle\lqn(K_n) > (1-\frac{1}{\sqrt{2}})n - \frac{1}{16}(9-4\sqrt{2})$.
\end{lemma}
\begin{proof} Consider a $k$-local queue layout $\calQ$ of $K_n$.  Without
loss of generality, each edge is contained in exactly one queue.  Moreover, the
vertices are ordered $v_1 \prec \cdots \prec v_n$ and the length of an edge $v_iv_j$
is defined as $|i-j|$.  Now for any edge $e = v_iv_j$ with $i < j$ consider
the queue $Q \in \calQ$ containing $e$.  We call~$e$ \emph{left-longest} if
there is no edge in $Q$ that is longer than $e$ and has the same right
endpoint as $e$, i.e., $Q$ contains no edge $v_{i'}v_j$ with $i' < i$.
Similarly, we call $e = v_iv_j \in Q$ \emph{right-shortest} if there is no
edge in $Q$ that is shorter than $e$ and has the same left endpoint as $e$,
i.e., $Q$ contains no edge $v_iv_{j'}$ with $i < j' < j$.  We have that
 \begin{enumerate}[(i)]
  \item every edge of $K_n$ is left-longest or right-shortest (or both).\label{enum:left-or-right-bad}
 \end{enumerate} In fact, if $v_iv_j \in Q$ is of neither type, then $Q$ would
contain two edges $v_{i'}v_j$ and $v_iv_{j'}$ with $i' < i < j' < j$, and
hence $Q$ would not be a queue.
 
For each vertex $v_i$ let $\ell_i$, respectively $r_i$, denote the number of
left-longest edges whose right endpoint is $v_i$, respectively the number of
right-shortest edges whose left endpoint is $v_i$.  That is,
 \begin{itemize}
  \item[] $\ell_i = \# \{ v_a \in V(K_n) \mid a < i \text{ and } v_av_i \text{
left-longest} \}$ and
  
  \item[] $r_i = \# \{ v_b \in V(K_n) \mid i < b \text{ and } v_iv_b \text{
right-shortest} \}$.
 \end{itemize} Further let $b_i$ denote the number of queues in $\calQ$ with
at least one edge whose right endpoint is $v_i$ \emph{and} at least one edge
whose left endpoint is $v_i$.  That is,
 \begin{itemize}
  \item[] $b_i = \# \{ Q \in \calQ \mid \exists a,b \text{ with } a < i < b
\text{ and } v_av_i,v_iv_b \in Q\}$.
 \end{itemize} We can then write the number of queues in $\calQ$ containing
 the vertex $v_i$ in terms of $\ell_i$, $r_i$ and $b_i$. Indeed, if $Q\in\calQ$
 contains an edge incident to $v_i$, then it contains a left-longest or a
 right-shortest or both, i.e., the contribution of $Q$ to $\ell_i + r_i - b_i$
 is exactly one.
 \begin{enumerate}[(i)]
  \setcounter{enumi}{1}
  \item Vertex $v_i$ has incident edges in exactly $\ell_i + r_i - b_i$ queues
in $\calQ$.\label{enum:incident-queues}
 \end{enumerate} 
 As every vertex is in at most $k$ queues, we have $b_i \leq k$ for
$i=1,\ldots,n$.  Also every vertex $v_i$ is the right endpoint of at most
$i-1$ edges and thus $b_i \leq i-1$.  Similarly, $v_i$ the left endpoint of at
most $n-i$ edges and thus $b_i \leq n-i$.  Together,
 \begin{enumerate}[(i)]
  \setcounter{enumi}{2}
  \item for every vertex $v_i$ we have $b_i \leq \min
\{i-1,n-i,k\}$.\label{enum:both-upper-bound}
 \end{enumerate}
 Using the above and assuming $ k \leq n/2 $ (we are done otherwise), we calculate
%
 \begin{align*} 
    kn 
    &\geq \sum_{i=1}^n \# \{ Q \in \calQ \mid v_i \in V(Q)\}
    \overset{\ref{enum:incident-queues}}{=} \sum_{i=1}^n (\ell_i + r_i - b_i)
    \overset{\ref{enum:left-or-right-bad}}{\geq} |E(K_n)| - \sum_{i=1}^n b_i \\
    &\overset{\ref{enum:both-upper-bound}}{\geq} \binom{n}{2} - \sum_{i=1}^k (i-1) - \sum_{i=n-k+1}^n (n-i) - (n-2k)k\\
    &= \binom{n}{2} - 2\binom{k}{2} - (n-2k)k.
 \end{align*}
 For $a\geq 1$ and $b>0$ we have
 \begin{equation}
  \sqrt{a+b} < \sqrt{a} + b/2. \label{eq:sqrt-bounds}
 \end{equation}
 Using this we get the desired
 bound for $k$ as follows. 
 \begin{align*}
   kn &\geq \binom{n}{2} - 2\binom{k}{2} - (n-2k)k \\
  \iff \quad 0 &\geq k^2 + (1-2n)k + \binom{n}{2}\\
  \implies \quad k &\geq (n-\frac{1}{2}) - \sqrt{ (n-\frac{1}{2})^2 - \binom{n}{2} }
  = (n-\frac{1}{2}) - \sqrt{\frac{1}{2}(n^2-n + \frac{1}{2})} \\
  &= (n-\frac{1}{2}) - \sqrt{\frac{1}{2}(n - \frac{1}{2})^2 + \frac{1}{8}}
  \overset{\eqref{eq:sqrt-bounds}}{>} (n-\frac{1}{2}) - \sqrt{\frac{1}{2}(n - \frac{1}{2})^2} - \frac{1}{16}\\
  &=  (1-\frac{1}{\sqrt{2}})(n-\frac{1}{2}) - \frac{1}{16} 
  = (1-\frac{1}{\sqrt{2}})n - \frac{1}{16}(9-4\sqrt{2})\qedhere 
  %
\end{align*}
\end{proof}



Now let us turn to the upper bound of the union queue number in \cref{thm:queue-number}.
We thereby also prove an upper bound for the local queue number.
However, we improve on this bound with a different construction in \cref{lem:local-queue-UB}.

\begin{lemma}
    \label{lem:uqn-UB}
    For any $ n \geq 0 $, we have 
    \[ \lqn(K_n) \leq \left\lceil 1-\frac{1}{\sqrt{2}} \right\rceil n + 11 \text{ and } \uqn(K_n) \leq \left\lceil 1-\frac{1}{\sqrt{2}} \right\rceil n + 42. \]
\end{lemma}

We prove that whenever $k \geq (1-1/\sqrt{2})(n + 1)$, there is a
$(k+11)$-local queue layout and a $ (k + 42) $-union queue layout of $K_{n+1}$.  
Let $v_1 \prec \dots \prec v_{n+1}$ be a
fixed vertex ordering of $K_{n+1}$.  For ease of presentation, we model the edge
set of $K_{n+1}$ as a point set $T_n$ in $\mathbb{Z}^2$ with triangular shape defined by
\[
  T_n = \{ (x, y) \in \mathbb{Z}^2 \mid x + y \leq n+1;\ x \geq 1;\ y \geq 1\}.
\]
The elements in $ T_n $ correspond to the entries of the adjacency matrix of $ K_{n + 1} $.
That is, element $(x, y)$ of $T_n$ corresponds to edge $ v_{n+2-y} v_x $ in $K_{n+1}$
and conversely edge $ v_i v_j $ in $K_{n+1}$ with $i > j$ corresponds to element
$(j, n+2-i)$ in $T_n$.  
Two edges $v_iv_j$ with $ i > j $ and $v_{i'}v_{j'}$
with $ i' > j'$ nest if and only if the corresponding elements $(j, n + 2 - i)$
and $(j', n + 2 - i')$ in $T_n$ are comparable in the strict dominance order of
$\mathbb{Z}^2$ (i.e.\ coordinate-wise strict inequalities of points).  
To see this, observe that small $ y $-coordinates correspond to a left endpoint having a small index, whereas small $ x $-coordinates correspond to a right endpoint with a large index.
Hence, an edge set $Q \subseteq E(K_{n+1})$ forms a queue if and only if the
corresponding points in $T_n$ form a weakly monotonically decreasing chain, see
\cref{fig:exampleQ}.

\begin{figure}
    \centering
    \includegraphics[height = 8em]{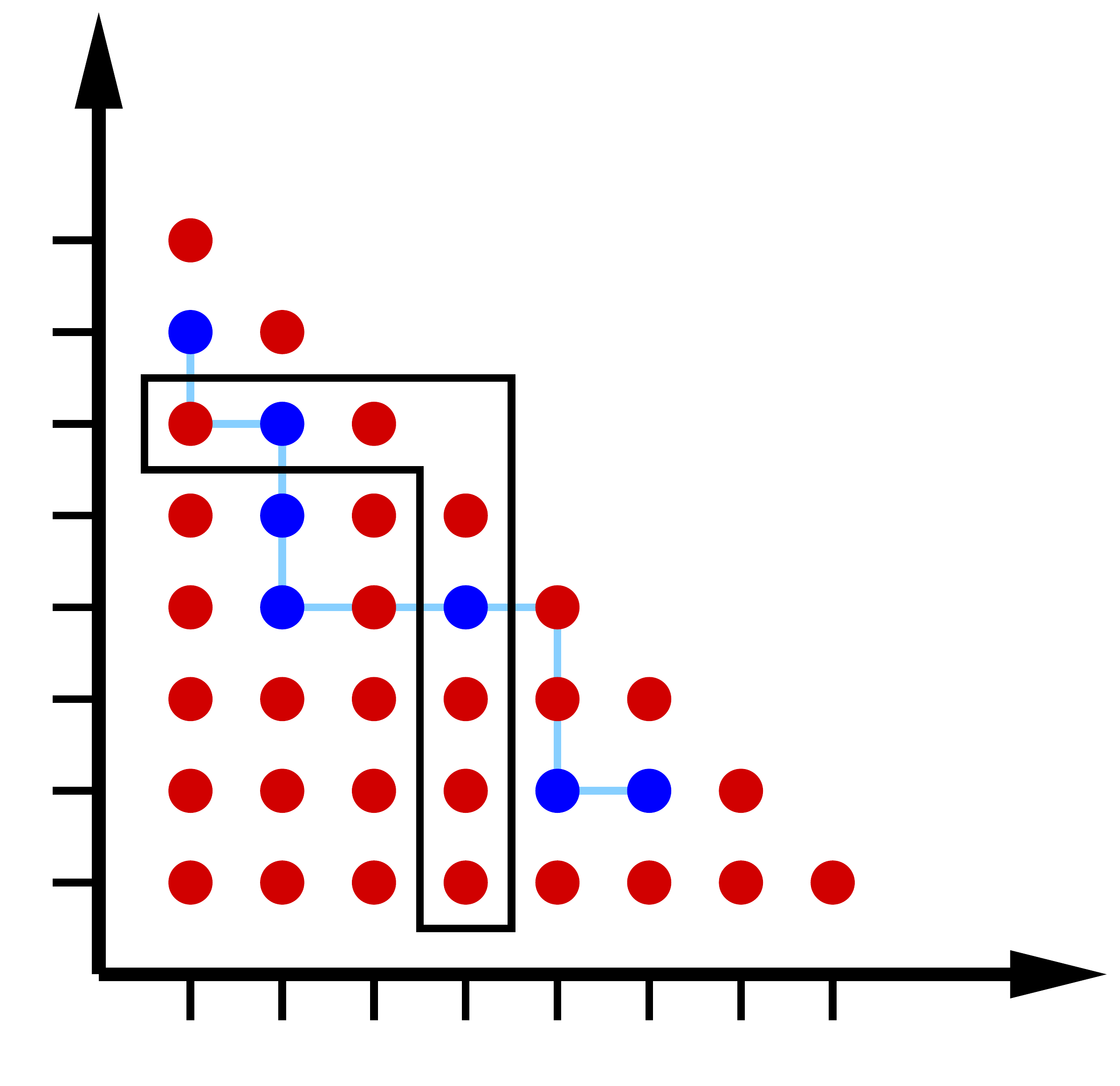}
    \hspace{2em}
    \includegraphics{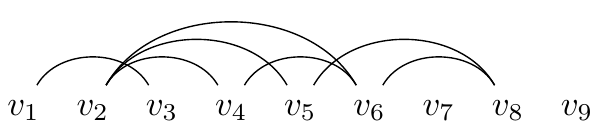}
    \caption{
        Left: Triangle $T_8$ corresponding to $K_9$ and the hook of vertex $4$.
        The blue entries represent a queue, the light blue zig-zag shows that the edges are non-nesting.
        Right: Linear layout of the blue queue.
    }
    \label{fig:exampleQ}
\end{figure}

A vertex $v_i$ of $K_{n+1}$ corresponds to column $i$ and row $n+2-i$ in
$T_n$.  We call the union of column $i$ and row $n+2-i$ the \emph{hook} of
vertex $v_i$.  If $H$ is the hook of vertex $v_i$ and $Q$ is a queue
corresponding to chain $C \subseteq T_n$, then vertex~$v_i$ is contained in
queue $Q$ if and only if $H \cap C \neq \emptyset$.
For our construction of a $(k+11)$-local queue assignment of $K_{n+1}$, we use the equivalent model of covering the triangular point set $T_n$ with monotone chains such that no hook intersects more than $k+11$ chains.

Analogously to union queues, we call a subset $ S \subseteq T_n $ a \emph{union chain} if there is a partition of $ S $ into weakly monotonically decreasing chains such that each hook intersects at most one of them.
To prove \cref{lem:uqn-UB}, we partition $ T_n $ into $ k + 42 $ union chains and therefore get a $ (k + 42) $-union queue layout for $ K_{n + 1} $.
 
\begin{lemma} 
    \label{lem:uqn-UB-triangle}
    For any integer $n \geq 0$ and any integer $k \geq (1-1/\sqrt{2})(n + 1)$, the points of $T_n$ can be partitioned into $ k + 42 $ union chains.
    In addition, the points of $ T_n $ can be partitioned into weakly monotonically decreasing chains such that each hook intersects at most $ k + 11 $ chains.
\end{lemma}
%
%
\begin{proof}
    First, we define weakly monotonically decreasing chains that cover $ T_n $ such that no hook intersects more than $ k + 11 $ chains.
    We then partition these chains into sets of chains that form the basis for our union chains.
    We assume that $ n $ is even and that $ k $ is the smallest even integer with $k \geq (1-1/\sqrt{2})(n + 1)$.
    To compensate for this assumption, we construct chains such that each hook intersects at most $ k + 9 $ chains and a partition of $ T_n $ into $ k + 40 $ union chains.
    We need $ 3n \geq 10k $ for the construction, which is the case for $ n \geq 294 $.
    For $ n \leq 56 $, we have $ \lfloor (n + 1)/2 \rfloor \leq \lceil 1-1/\sqrt{2} \rceil (n + 1) + 11 $, so an $ \lfloor (n + 1)/2 \rfloor $-queue layout of $ K_{n + 1} $ gives the desired partition of $ K_{n + 1} $, respectively $ T_n $, into queues, respectively chains.
    For any $ n $ between $ 56 $ and $ 294 $, one can check that the desired bounds can be obtained from a queue layout of some $ K_{n'} $ with a slightly smaller or greater number of vertices.%
    \footnote{%
        Indeed, for each $ n $ between $ 56 $ and $ 294 $ and corresponding $k$, one of the following two cases applies.
        First, there is an even $ n' \geq n $ such that $ k' = \lceil (1-1/\sqrt{2})(n'+1) \rceil $ is even, $ 3 n' \geq 10 k' $ holds, and we have 
        $ \lqn(K_n) \leq \lqn(K_{n'}) \leq k' + 9 \leq k + 11 $, respectively 
        $ \uqn(K_n) \leq \uqn(K_{n'}) \leq k' + 40 \leq k + 42 $.
        Or second, there is an even $ n' $ such that 
        $ n - 4 \leq n' \leq n $, 
        $ k' = \lceil (1-1/\sqrt{2})(n'+1) \rceil $ is even, 
        $ 3 n' \geq 10 k' $ holds, and 
        the queue layout we obtain for $ K_{n'} $ can be augmented to a queue layout of $ K_n $ matching the desired bounds.
        To do so, we attach at most two of the additional $ n' - n \leq 4 $ vertices to the left and at most two to the right and use two additional queues to cover the new edges.
        We then observe that $ k' + 2 \leq k $, which gives the desired bounds for $ K_n $.
    }

    We start by defining a family $ \mathcal{L} $ of $k$ chains $L_1,\ldots,L_k$, illustrated in \cref{fig:diagonal-chains}.
    Chain $L_i$ is composed of three blocks. The first block consists of
    the $2(k-i+1)$ topmost elements in column~$i$ of $T_n$.
    The second block starts at the lowest element of the first block,
    continues with a right and down alternation for $2(n-2(k-1))$
    steps, and ends in row $ i $.
    The last block consists of the
    $2(k-i+1)$ rightmost elements in row~$ i $.  
    Formally, for $ i = 1,\ldots,k $ we set
    \begin{align*} 
        L_i = &\, \{ (i, y) \in T_n \mid n+1-i \geq y \geq n-2k+i \} \\
        \cup &\,\{ (x, y) \in T_n \mid x, y \geq i \text{ and } n - 2(k - i) \leq x + y \leq n - 2 (k - i) + 1 \} \\
        \cup&\,\{ (x, i) \in T_n \mid n-2k + i \leq x \leq n+1 -i\}.
    \end{align*} 
    The chains of $ \mathcal{L} $ cover all points of $ T_n $ except for the bottom left triangle $ T_{n - 2k} $.
    The remaining points are covered by chains containing only points of a single column or row.
    We refer to these chains as \emph{vertical} and \emph{horizontal} chains, respectively.
    Note that vertical and horizontal chains correspond to stars in $ K_{n + 1} $.
    We only define families $ \mathcal{A}, \dots, \mathcal{G} $ containing vertical chains, the horizontal chains $ \mathcal{A'}, \dots, \mathcal{G'} $ are then defined symmetrically, i.e., $ \mathcal{M}' = \{ (y, x) \in T_n \mid (x, y) \in \mathcal{M} \} $ for $ \mathcal{M} = \mathcal{A}, \dots, \mathcal{G} $.
    The resulting layout of $ T_{n - 2k} $ is illustrated in \cref{fig:small-triangle-local}.
    
    \begin{figure}[tbp]
        \centering
        \includegraphics[width=\textwidth]{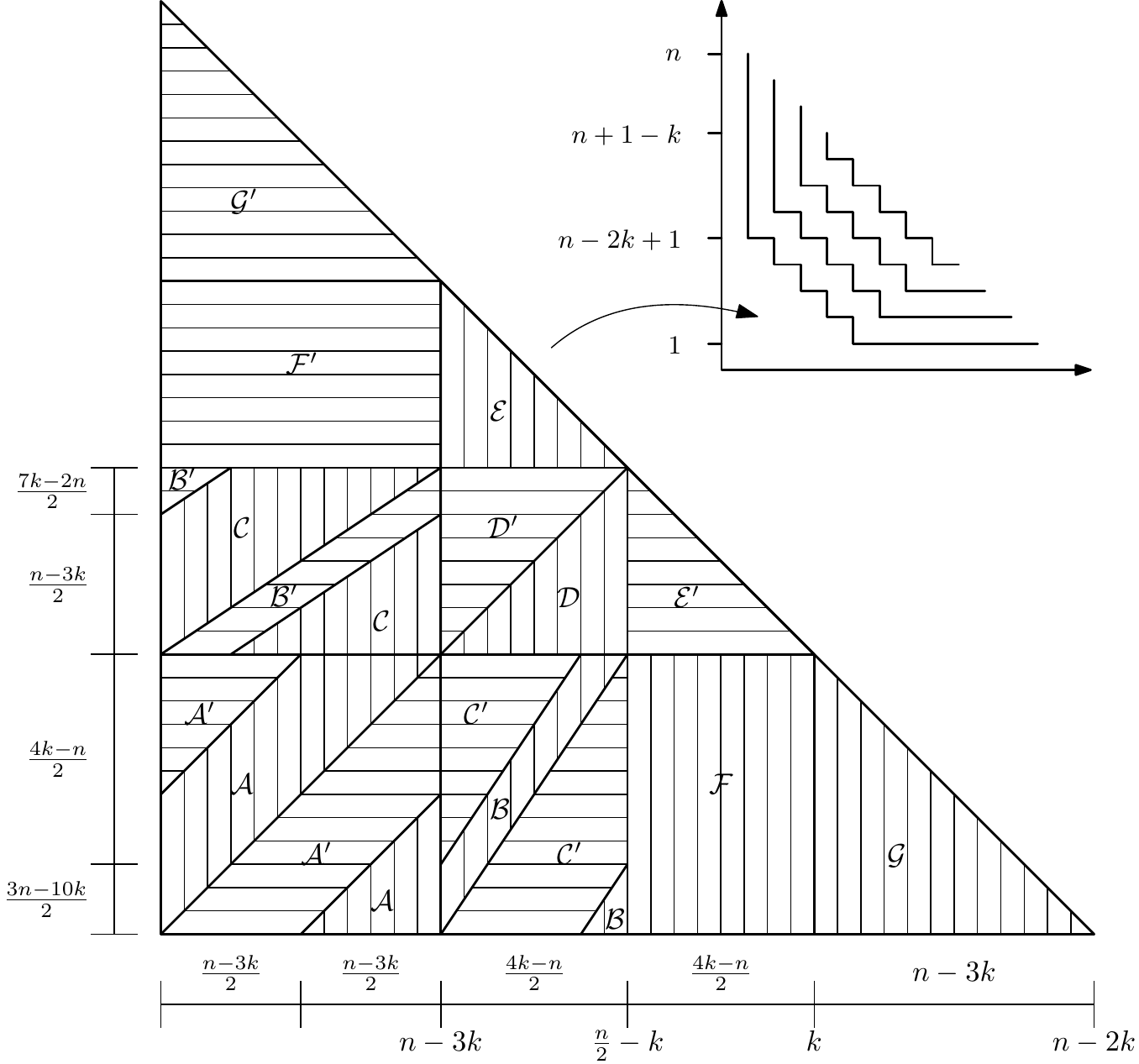}
        \caption{
            Chains $ L_1, \dots, L_k $ (top right). The triangle $ T_{n - 2k} $ is covered by families $ \mathcal{A}, \dots, \mathcal{G} $ and $ \mathcal{A'}, \dots, \mathcal{G'} $ of vertical, respectively horizontal, chains (bottom left).}
        \label{fig:small-triangle-local}
        \label{fig:diagonal-chains}
    \end{figure}

    The families $ \mathcal{A}, \dots, \mathcal{D} $ cover the bottom left square $ S_{n/2 - k} = \{ (x, y) \in T_n \mid 1 \leq x, y \leq n/2 - k \} $.
    For this, let $ \mathcal{A} $ consist of chains $ A_1, \dots, A_{n - 3k} $ and $ \hat{A}_1, \dots, \hat{A}_{n - 3k} $ with $ A_i \cup \hat{A}_i $ having size $ (n - 3k)/2 $ or $ (n - 3k)/2 + 1 $ for $ i = 1, \dots, n - 3k $.
    The chains $ A_i $ consist of points in column $ i $ starting with the bottommost point in row $ i $.
    Chains $ A_i $ with points above $ y = n - 3k $ continue as $ \hat{A}_i $ from the bottom.
    Note that $ \hat{A}_i $ is empty for $ i \leq (n - 3k)/2 $.
    For $ i = 1, \dots, n - 3k $ we define
    \begin{align*}
        A_i        & = \{ (i, y) \in T_n \mid i \leq y < (n - 3k)/2 + i \text{ and } y \leq n - 3k\} \text{ and} \\
        \hat{A}_i  & = \{ (i, y) \in T_n \mid 1 \leq y \leq i - (n - 3k)/2 \}.
    \end{align*}
    The defined chains together with the chains in $ \mathcal{A}' $ cover all points whose $ x $- and $ y $-coordinates are at most $ n - 3k $.
    Some points, however, are covered twice. 
    We may choose any of the two covering chains for the respective points to obtain a partition.

    Family $ \mathcal{B} $ is located to the right of $ \mathcal{A} $ and consists of chains $ B_1, \dots, B_{(4k - n)/2} $ and $ \hat{B}_1, \dots, \hat{B}_{(4k - n)/2} $.
    Above $ \mathcal{A} $, there are chains $ C_1, \dots, C_{n - 3k} $ and $ \hat{C}_1, \dots,\allowbreak \hat{C}_{n - 3k} $ forming family $ \mathcal{C} $.
    Considering the $ y $-coordinates of the bottommost points of the chains $ B_i $ in $ \mathcal{B} $, we have a slope $ s = \frac{n - 3k}{(4k - n)/2} $ ($ = \sqrt{2} $ for $ k = (1 - 1/\sqrt{2}) n $).
    Symmetrically, the slope is $ 1/s $ in $ \mathcal{C} $.
    For $ i = 1, \dots, (4k - n)/2 $ and $ j = 1, \dots, n - 3k $ we define
    \begin{align*}
        B_i        & = \{ (n - 3k + i, y) \in T_n \mid \lfloor si \rfloor \leq y \leq \lceil si \rceil + (3n - 10k)/2 \text{ and } y \leq n - 3k \}, \\
        \hat{B}_i  & = \{ (n - 3k + i, y) \in T_n \mid 1 \leq y \leq \lceil si \rceil + (3n - 10k)/2 - (n - 3k) \}, \\
        C_j        & = \{ (j, y) \in T_n \mid n - 3k + \lfloor j/s \rfloor \leq y \leq 3(n - 3k)/2 + \lceil j/s \rceil \text{ and } \\
                        & \qquad n - 3k < y \leq n - 2k \}, \text{ and} \\
        \hat{C}_j  & = \{ (j, y) \in T_n \mid n - 3k < y \leq 3(n - 3k)/2 + \lceil j/s \rceil - (n - 2k) \}.
    \end{align*}
    Note that $ \hat{B}_i $ and $ \hat{C}_j $ are empty for $ i \leq (n - 3k)/2 $ and $ j \leq (3n - 10k)/2 $.
    The chains in $ \mathcal{B} \cup \mathcal{C'} $, and symmetrically $ \mathcal{B'} \cup \mathcal{C} $, form a rectangle containing $ (n - 3k)(4k - n)/2 $ points.
    Again, we choose any chain for points that are covered by multiple chains.

    Family $ \mathcal{D} $, together with the corresponding horizontal chains in $ \mathcal{D'} $, accomplishes the square $ S_{n/2 - k} $.
    For this, we have $ (4k - n)/2 $ chains $ D_1, \dots, D_{(4k - n)/2} $, where $ D_i $ has size $ i $.
    For $ i = 1, \dots, (4k - n)/2 $ let
    \begin{align*}
        D_i & = \{ (n - 3k + i, y) \in T_n \mid n - 3k < y \leq n - 3k + i \}.\\
    \end{align*}
    Finally, we cover to remaining two triangles with three families $ \mathcal{E}, \mathcal{F} $, and $ \mathcal{G} $ of vertical chains.
    Families $ \mathcal{E} $ and $ \mathcal{G} $ consist of chains $ E_1, \dots, E_{(4k - n)/2} $, respectively $ G_1, \dots, G_{n - 3k} $, each filling a triangle,
    whereas the chains $ F_1, \dots, F_{(4k - n)/2} \in \mathcal{F} $ all have size $ n - 3k $ and cover a rectangle.
    For $ i = 1, \dots, (4k - n)/2 $ and $ j = 1, \dots, n - 3k $ we define

    \begin{align*}
        E_i & = \{ (n - 3k + i, y) \in T_n \mid n/2 - k < y \leq k - (i - 1) \}, \\
        F_i & = \{ (n/2 - k + i, y) \in T_n \mid 1 \leq y \leq n - 3k \}, \text{ and} \\
        G_j & = \{ (k + j, y) \in T_n \mid 1 \leq y \leq n - 3k - (j - 1) \}.
    \end{align*}
    The chains in $ \mathcal{E}' $, $ \mathcal{F} $, and $ \mathcal{G} $ together cover the all points $ (x, y) $ of $ T_{n - 2k} $ with $ x > n - 3k $, and symmetrically $ \mathcal{E} $, $ \mathcal{F'} $, and $ \mathcal{G'} $ cover the triangle above $ S_{n/2 - k} $.
    Note that we have $ n/4 \leq 2n/7 \leq (1 - 1/\sqrt{2})(n + 1) \leq k $ and recall that $ n \geq 10k/3 \geq 3k $ so all intervals are well-defined.

    Next we show that each hook intersects at most $ k + 9 $ chains.
    Recall that a hook of a vertex is the union of the row and the column representing that vertex.
    We first count for each row $ y = 1, \dots, n $ and each column $ x = 1, \dots, n $ the number of intersecting chains and then add up the results to obtain the number of intersecting chains for each hook.
    We start by counting the vertical chains, which intersect rows $ 1, \dots, k $.
    Note that no vertical chain intersects any row above $ y = k $.
    If $ n = 3k $, then $ \mathcal{B} $ and $ \mathcal{C} $ are empty, so we may assume $ s > 0 $ when counting these chains.
    \begin{itemize}
        \item [$ \mathcal{A} $:]
            For each $ y = 1, \dots, (n - 3k)/2 $, row $ y $ intersects chains $ A_1, \dots, A_y $ and $ \hat{A}_{(n - 3k)/2 + y}, \dots,\allowbreak \hat{A}_{n - 3k} $, which sums up to $ (n - 3k)/2 + 1 $ chains.
            The chains $ \hat{A}_1, \dots, \hat{A}_{n - 3k} $ do not contain any points above row $ (n - 3k)/2 $ and thus row $ y = (n - 3k)/2 + 1, \dots, n - 3k $ intersects exactly $ A_{y - (n - 3k)/2 + 1}, \dots A_{y} $, i.e., $ (n - 3k)/2 $ chains.
        \item [$ \mathcal{B} $:]
            For the upcoming calculation, note that $ (4k - n)^2 \geq 2 (n - 3k)^2 $ for $ k \geq (1 - 1/\sqrt{2}) n $.
            For row $ y = 1, \dots, (3n - 10k)/2 $, we have chains 
            $ B_i $ for $ i \leq (y + 1)/s $ and
            $ \hat{B}_i $ for $ (4k - n)/2 \geq i > \frac{y - 1 - (n - 3k) - (3n - 10k)/2}{s} = \frac{y - 1 + (4k - n)/2}{s} $.
            This sums up to 
            \begin{multline*} 
                \frac{y + 1}{s} + \frac{4k-n}{2} - \frac{y - 1 + (4k - n)/2}{s}
                = \frac{4k-n}{2} - \frac{(4k - n)/2}{s} + \frac{2}{s} \\
                = \frac{(4k - n)}{2} - \frac{(4k - n)^2}{2 (n - 3k)^2} \frac{n - 3k}{2} + \frac{2}{s}
                \leq \frac{(4k - n)}{2} - \frac{n - 3k}{2} + \frac{2}{s}
                \overset{(\ast)}{<} \frac{7k - 2n}{2} + 4 
            \end{multline*}
            chains.
            For $ y = (3n - 10k)/2 + 1, \dots, n - 3k $, each row intersects at most 
            \begin{multline*} 
                \frac{y + 1}{s} - \frac{y - (3n - 10k)/2 - 1}{s} 
                = \frac{3n - 10k}{2s} + \frac{2}{s} 
                = (n - 3k - \frac{4k - n}{2})/s + \frac{2}{s} \\
                = \frac{4k - n}{2} - \frac{(4k - n)^2}{2 (n - 3k)^2} \frac{n - 3k}{2} + \frac{2}{s}
                \leq \frac{4k - n}{2} -  \frac{n - 3k}{2} + \frac{2}{s}
                \overset{(\ast)}{<} \frac{7k - 2n}{2} + 4
            \end{multline*}
            chains.
            For $ (\ast) $, we remark that $ 2/s < 4 $ holds for all $ n \geq 104 $.
            For smaller $ n $, recall that there are only $ (4k - n)/2 $ columns containing vertical chains in $ \mathcal{B} $.
            That is, if $ (7k - 2n)/2 + 4 > (4k - n)/2 $, then $ (7k - 2n)/2 + 4 $ certainly upper-bounds the number of chains in $ \mathcal{B} $ that intersect any row.
            However, if $ n < 104 $ and $ (7k - 2n)/2 + 4 < (4k - n)/2 $, then we also have $ 2/s < 4 $.
        \item [$ \mathcal{C} $:]
            For $ y = 1, \dots, (n - 3k)/2 $, chain $ C_i $ intersects row $ n - 3k + y $ only if $ i \leq (y + 1)s $ and $ \hat{C}_i $ intersects row $ n - 3k + y $ only if $ n - 3k \geq i > (y - 1 - (7k - 2n)/2) $.
            So we have at most
            \begin{multline*}
                (y + 1)s + n - 3k - (y - 1 + (7k - 2n)/2)s \\
                = 2s + \frac{2 (n - 3k)^2}{(4k - n)^2} \frac{4k - n}{2} 
                \leq 2s + \frac{4k - n}{2} 
                \leq \frac{4k - n}{2} + 3 
            \end{multline*}
            chains.
            For $ y = (n - 3k)/2 + 1, \dots, (4k - n)/2 $, we have $ C_i $ in row $ n - 3k + y $ for $ (y - 1 - (n - 3k)/2)s < i < (y + 1)s $, which upper-bounds the number of chains by 
            \begin{equation*}
                ((n - 3k) + 2)s 
                = \frac{2 (n - 3k)^2}{(4k - n)^2} \frac{4k - n}{2} + 2s 
                \leq \frac{4k - n}{2} + 3.
            \end{equation*}
        \item [$ \mathcal{D} $:]
            The vertical chains in $ \mathcal{D} $ intersect the rows $ n - 3k + 1, \dots, n/2 - k $.
            Precisely, for $ y = 1, \dots, (4k - n)/2 $, row $ n - 3k + y $ intersects the chains $ D_y, \dots, D_{(4k - n)/2} $, i.e., $ (4k - n)/2 - (y - 1) $ chains.
        \item [$ \mathcal{E} $:]
            Similarly, the vertical chains in $ \mathcal{E} $ intersect the rows $ n/2 - k + 1, \dots, k $.
            For each $ y = 1, \dots, (4k - n)/2 $, row $ n/2 - k + y $ intersects the chains $ E_1, \dots, E_{(4k - n)/2 - (y - 1)} $, i.e., $ (4k - n)/2 - (y - 1) $ chains.
        \item [$ \mathcal{F} $:]
            All chains of $ \mathcal{F} $ intersect the rows $ 1, \dots, n - 3k $ exactly once, that is we have $ (4k - n)/2 $ additional vertical chains in these rows.
        \item [$ \mathcal{G} $:]
            Row $ y = 1, \dots, n - 3k $ intersects the chains $ G_1, \dots, G_{n - 3k - (y - 1)} $, i.e., $ n - 3k - (y - 1) $ chains.
    \end{itemize}
    Summing up the number of vertical chains intersecting the rows $ y = 1, \dots, n - 3k $, we get at most $ k + 5 - y $ chains from $ \mathcal{A}, \mathcal{B}, \mathcal{F}, $ and $ \mathcal{G} $.
    Additionally we have at most four horizontal chains in $ \mathcal{A}' $ and $ \mathcal{C}' $. 
    Rows $ y = n - 3k + 1, \dots, n/2 - k $ intersect vertical chains in $ \mathcal{C} $ and $ \mathcal{D} $, which sums up to $ k - y + 4 $ chains.
    There are at most four additional horizontal chains in $ \mathcal{B}', \mathcal{D}', $ and $ \mathcal{E}' $ intersecting these rows.
    Finally, rows $ y = n/2 - k + 1, \dots, k $ intersect only vertical chains in $ \mathcal{E} $ and horizontal chains in $ \mathcal{F}' $, summing up to $ k - y + 2 $ chains.
    Together, we have at most $ k - y + 9 $ vertical and horizontal chains in row $ y = 1, \dots, k $ and only one horizontal chain from $ \mathcal{G}' $ in row $ k + 1, \dots, n - 2k $.
    Symmetrically, there are at most $ k - x + 9 $ vertical and horizontal chains intersecting column $ x = 1, \dots, k $ and one vertical chain in each column $ k + 1, \dots, n - 2k $.

    Now, we count the number of chains intersecting the hook $ H_x $ of vertex $ v_x $ for $ x = 1, \dots, n + 1 $.
    Recall that hook $ H_x $ corresponds to column $ x $ and row $ n + 2 - x $.
    For $ x = 1, \dots, k $, hook $ H_x $ intersects $ L_1, \dots, L_x $ and $ k - x + 9 $ vertical and horizontal chains.
    Note that row $ n + 2 - x $ is above row $ n - 2k $ and thus does not intersect any vertical or horizontal chains in $ T_{n - 2k} $.
    Hooks $ H_{k + 1}, \dots, H_{n/2 + 1} $ intersect all $ k $ chains in $ \mathcal{L} $ and at most one vertical chain from $ \mathcal{G} $.
    By symmetry all hooks intersect at most $ k + 9 $ chains.

    The chains defined above induce a $ (k + 9) $-local queue layout.
    To obtain a $ (k + 40) $-union queue layout, we partition the set of all chains into sets of chains $ \mathcal{S}_1, \dots, \mathcal{S}_{k + 6} $ with $ L_i \in \mathcal{S}_i $ for $ i = 1, \dots, k $.

    We first introduce some notions that allow us to transform a set of chains into a union chain.
    Consider some set $ \mathcal{S} = \{ S_1, \dots, S_m \} $ of chains.
    Note that $ \bigcup_{i \in [m]} S_i $ is not necessarily a union chain as two chains may intersect the same hook.
    We call the set of all hooks that intersect at least two chains of $ \mathcal{S} $ the \emph{common hooks} of $ \mathcal{S} $.
    If $ \mathcal{S} $ has no common hooks, then $ \mathcal{S} $ is already a union chain.
    Otherwise, we assign each of the common hooks to at most one chain.
    A point that is contained in some chain $ S \in \mathcal{S} $ and in some common hook that is not assigned to $ S $ is called a \emph{bad point}.
    Removing all bad points yields chains that together form a union chain.

    We now aim to define $ \mathcal{S}_1, \dots, \mathcal{S}_{k + 6} $ such that the resulting bad points can be covered by a constant number of union chains.
    For this, we associate each vertical (horizontal) chain $ C $ with the interval $ I_C \subseteq [k] $ that consists of the $ y $-coordinates ($ x $-coordinates) of the points contained in $ C $.
    We say two vertical (horizontal) chains \emph{overlap} if the corresponding intervals are not disjoint.
    Consider the set $ \mathcal{S}_i $ for some $ i = 1, \dots, k + 6 $.
    We add vertical and horizontal chains to $ \mathcal{S}_i $ such that 
    \begin{enumerate}[(i)]
    \item \label[cond]{cond:overlap} 
        chains in $ \mathcal{S}_i $ do not overlap,
    \item \label[cond]{cond:left-bottom-chains}
        the $ y $-coordinates ($ x $-coordinates) of all points in vertical (horizontal) chains in $ \mathcal{S}_i $ are smaller than $ i $, and
    \item \label[cond]{cond:L_i-conflict}
        there is no vertical (horizontal) chain in $ \mathcal{S}_i $ in column (row) $ i $.
    \end{enumerate}
    We first assume that \cref{cond:overlap,cond:left-bottom-chains,cond:L_i-conflict} hold and show that they can indeed be satisfied at the end of the proof. 
    We merge vertical (horizontal) chains of $ \mathcal{S}_i $ that are in the same column (row) into a single chain.
    Next, we assign common hooks to chains and use the three conditions to show that $ 34 $ union chains suffice to cover all bad points.
    For the analysis of bad points, we concentrate on vertical chains. 
    The result for horizontal chains follows symmetrically.

    \begin{figure}
        \centering
        \includegraphics{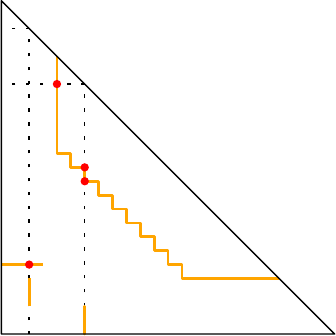}
        \caption{A set $ \mathcal{S}_i $ with two vertical chains and their hooks (dashed). Bad points are marked red.}
        \label{fig:bad-points}
    \end{figure}

    Consider a vertical chain $ C \in \mathcal{S}_i $ in column $ x $ and let $ H_C $ denote the hook that contains $ C $.
    Recall that vertical chains of $ \mathcal{S}_i $ that are in the same column are merged into a single chain.
    If $ H_C $ is a common hook, we assign it to $ C $. 
    Note that each hook either contains vertical chains or horizontal chains, and thus no hook is assigned to multiple chains.
    See \cref{fig:bad-points} for an illustration of bad points in $ \mathcal{S}_i $.
    Assigning $ H_C $ to $ C $ implies that all points of $ H_C $ that are contained in some other chain of $ \mathcal{S}_i $ are bad points.
    We say that these bad points are \emph{caused by $ C $}.
    We now analyze how many bad points are caused by $ C $.
    If $ x > i $, then $ H_C $ intersects $ L_i $ in at most four points, i.e., we have at most four bad points in $ H_C \cap L_i $.
    In this case, no horizontal chain in $ \mathcal{S}_i $ intersects $ H_C $ by \cref{cond:left-bottom-chains}.
    Otherwise we have $ x < i $ (due to \cref{cond:L_i-conflict}), and \cref{cond:overlap} ensures that $ H_C $ intersects at most one horizontal chain $ C' \in \mathcal{S}_i $.
    We thus have at most one bad point in $ H_C \cap C' $.
    In either case, $ C $ causes at most four bad points.
    
    It is left to show that each bad point is caused by some vertical or horizontal chain.
    For this, we keep $ \mathcal{S}_i $ fixed and consider a hook $ H $ whose column $ x $ intersects the triangle $ T_{n - 2k} $, i.e., with $ x \leq n - 2k $.
    We show that each bad point in $ H $ is caused by a vertical chain.
    The result for hooks whose row intersect $ T_{n - 2k} $ follows symmetrically with horizontal chains.
    Hooks that do not intersect $ T_{n - 2k} $ contain no bad points since $ L_i $ is the only chain in $ \mathcal{S}_i $ that intersects such a hook.
    Note that the row of $ H $ is above $ T_{n - 2k} $ and thus does not intersect any vertical or horizontal chains.
    If $ x < i $, then $ H $ does not intersect $ L_i $.
    It intersects at most one horizontal chain due to \cref{cond:overlap}.
    That is, if $ H $ contains a bad point, then it also contains a vertical chain causing this bad point.
    If $ x = i $, then $ H $ intersects neither vertical nor horizontal chains by \cref{cond:left-bottom-chains,cond:L_i-conflict}.
    Thus, $ L_i $ is the only intersected chain in $ \mathcal{S}_i $ and there are no bad points.
    If $ x > i $, then $ H $ intersects no horizontal chains in $ \mathcal{S}_i $ by \cref{cond:left-bottom-chains}.
    Hence, if there are bad points, then they are caused by a vertical chain in column $ x $.

    We are now ready to cover the bad points by a constant number of union chains.
    Let $ G \subseteq K_{n + 1} $ denote the graph that is induced by all bad points.
    For a vertical or horizontal chain $ C \subseteq T_{n - 2k} $, let $ v_C \in V(K_{n + 1}) $ denote the vertex that is represented by hook $ H_C $.
    We orient the edges of $ G $ such that every edge whose corresponding bad point is caused by chain $ C $ is oriented away from $ v_C $.
    Each hook contains at most four vertical or horizontal chains (see \cref{fig:small-triangle-local} and recall that each hook either contains vertical chains or horizontal chains).
    Recall that
    each chain causes at most four bad points.
    Thus, the out-degree of every vertex of $ G $ is at most $ 16 $.
    By \cref{prop:uqn-mad}, the graph $ G $ can be covered with $ \mad(G) + 2 \leq 2 \cdot 16 + 2 = 34 $ union queues using an arbitrary vertex ordering.
    Hence, the bad points can be covered by $ 34 $ union chains.

    Finally, we show how to partition the vertical chains of the presented $ (k + 9) $-local layout such that \cref{cond:overlap,cond:left-bottom-chains,cond:L_i-conflict} are satisfied.
    Let $ H $ denote the interval graph that is given by the intervals that correspond to vertical chains, 
    i.e., $ V(H) = \{ I_C \subseteq [k] \mid C \in \mathcal{A} \cup \dots \cup \mathcal{G} \} $ and there is an edge between two vertices if and only if the intervals are not disjoint.
    A clique of $ m $ vertices in $ H $ corresponds to a row that intersects $ m $ vertical chains.
    Recall that every row $ y = 1, \dots, k $ intersects at most $ k - y + 5 $ vertical chains.
    In particular, the clique number of $ H $ is at most $ k + 4 $.
    
    Note that any proper $ (k + 6) $-coloring and an arbitrary mapping between color classes and the sets of chains $ \mathcal{S}_1, \dots, \mathcal{S}_{k + 6} $ satisfies \cref{cond:overlap}.
    We next find a coloring with less than $ k + 6 $ colors that also satisfies \cref{cond:left-bottom-chains,cond:L_i-conflict}.
    We define an ordering on the vertices of $ H $ by decreasing topmost points of the intervals, i.e., $ [a, b] \prec [a', b'] $ if and only if $ b > b' $ or $ b = b' $ and $ a > a' $.
    We color the vertices of $ H $ greedily with $ k + 5 $ colors $ k + 6, \dots, 2 $.
    That is, for an interval in column $ x $, we choose the largest color that is not used by any smaller neighbor and that does not equal $ x $.
    We then define $ \mathcal{S}_i $ to contain the vertical chains whose intervals have color $ i $.
    Since $ H $ is an interval graph, the set consisting of a vertex $ [a, b] \in V(H) $ and its smaller neighbors induces a clique that corresponds to row $ b $.
    The vertex $ [a, b] $ thus has at most $ k - b + 4 $ smaller neighbors.
    There are at least two colors left that are larger than $ b $ and that are not already used by a smaller neighbor.
    Choosing one that does not equal $ x $ satisfies \cref{cond:left-bottom-chains,cond:L_i-conflict}.

    By symmetry, we color the horizontal chains with colors $ k + 6, \dots, 2 $ such that \cref{cond:overlap,cond:left-bottom-chains,cond:L_i-conflict} hold.
    We now have a partition of all chains into $ k + 6 $ sets of chains $ \mathcal{S}_1, \dots, \mathcal{S}_{k + 6} $, each forming a union chain when bad points are removed.
    Together with the $ 34 $ union chains for bad points, we get a partition of $ T_n $ into $ k + 6 + 34 = k + 40 $ union chains.
\end{proof}

\Cref{lem:uqn-UB} provides an upper bound of $ (1-1/\sqrt{2})n + \mathcal{O}(1) $ both for the union queue number and the local queue number and thus completes the proof of \cref{thm:queue-number}. 
For the local queue number, however, the following lemma improves on the constant of the additive term such that it almost matches the lower bound.

\begin{lemma}\label{lem:local-queue-UB}
  For any $n$, we have
  $\displaystyle\lqn(K_n) \leq \left\lceil(1-\frac{1}{\sqrt{2}})\,n \right\rceil + 1.$
\end{lemma}

For our example of a $(k+1)$-local queue assignment $\calQ$ of $K_{n+1}$, we again
use the equivalent model of covering the triangular point set $T_n$ with a set
$\calF$ of monotone chains such that no hook intersects more than $k+1$ of the
chains in~$\calF$.
 
\begin{lemma} For any $n \geq 0$ and any $k \geq (1-1/\sqrt{2})(n +1)$, the
elements of $T_n$ can be partitioned into a set $\calF_n$ of weakly monotonously
decreasing chains in such a way that
 \begin{enumerate}[(i)]
  \item every hook intersects at most $k+1$ chains in $\calF_n$,
  \item for each $i \leq k+1$, column $n+1-i$ intersects at most $i$ chains in
$\calF_n$,\label{enum:column}
  \item for each $i \leq k+1$, row $n+1-i$ intersects at most $i$ chains in
$\calF_n$.\label{enum:row}
 \end{enumerate}
\end{lemma}
\begin{proof}
  Note that conditions \ref{enum:column} and \ref{enum:row} are true for every
  partition of $T_n$ with chains, simply because there are only $i$ entries in
  the respective column and row.

  We proceed by induction on $n$.  For $n \leq 8$ the value of $k+1$
  is at least as large as the queue number of $K_{n+1}$, this yields
  an induction base.
 
 \medskip
 
 For the remainder assume that $n > 8$ and (without loss of generality) that
$k = \lceil (1-1/\sqrt{2})(n+1)\rceil$.  In particular, we have
\begin{equation}
  (1-1/\sqrt{2})(n+1) \leq k < (1-1/\sqrt{2})n + 1.
  \label{eq:queue-bounds-on-k}
\end{equation}
For $n \geq 6$ we therefore have $2k \leq n$.

The family $\calF_n$ will be composed of families $\calA$, $\calB$, $\calC$
and a family $\calF_{n'}$ for suitable $n'$ which is provided by induction.
Figure~\ref{fig:schema68} shows an example.

Family $\calA$ consists of $k$ chains $A_1,\ldots,A_k$. Chain $A_\alpha$
is composed of three blocks. The first block consists of
the $2(k-\alpha+1)$ topmost elements in column~$\alpha$ of
$T_n$, the second block starts at the lowest element of the first block
and continues with a right and down alternation with $2(n-2(k-1))$
steps, it ends in row $\alpha$, and the last block consists the of
$2(k-\alpha+1)$ rightmost elements in row $\alpha$.  Formally, for
$\alpha = 1,\ldots,k$ we set
\begin{align*} 
    A_\alpha = &\, \{ (\alpha, y) \in T_n \mid n+1-\alpha \geq y \geq n-2k+\alpha \} \\
    \cup &\,\{ (x, y) \in T_n \mid x, y \geq \alpha \text{ and } n - 2(k - \alpha) \leq x + y \leq n - 2 (k - \alpha) + 1 \} \\
    \cup&\,\{ (x, \alpha) \in T_n \mid n-2k + \alpha \leq x \leq n+1 -\alpha\}.
\end{align*}

\begin{figure}
    \includegraphics[width = \textwidth]{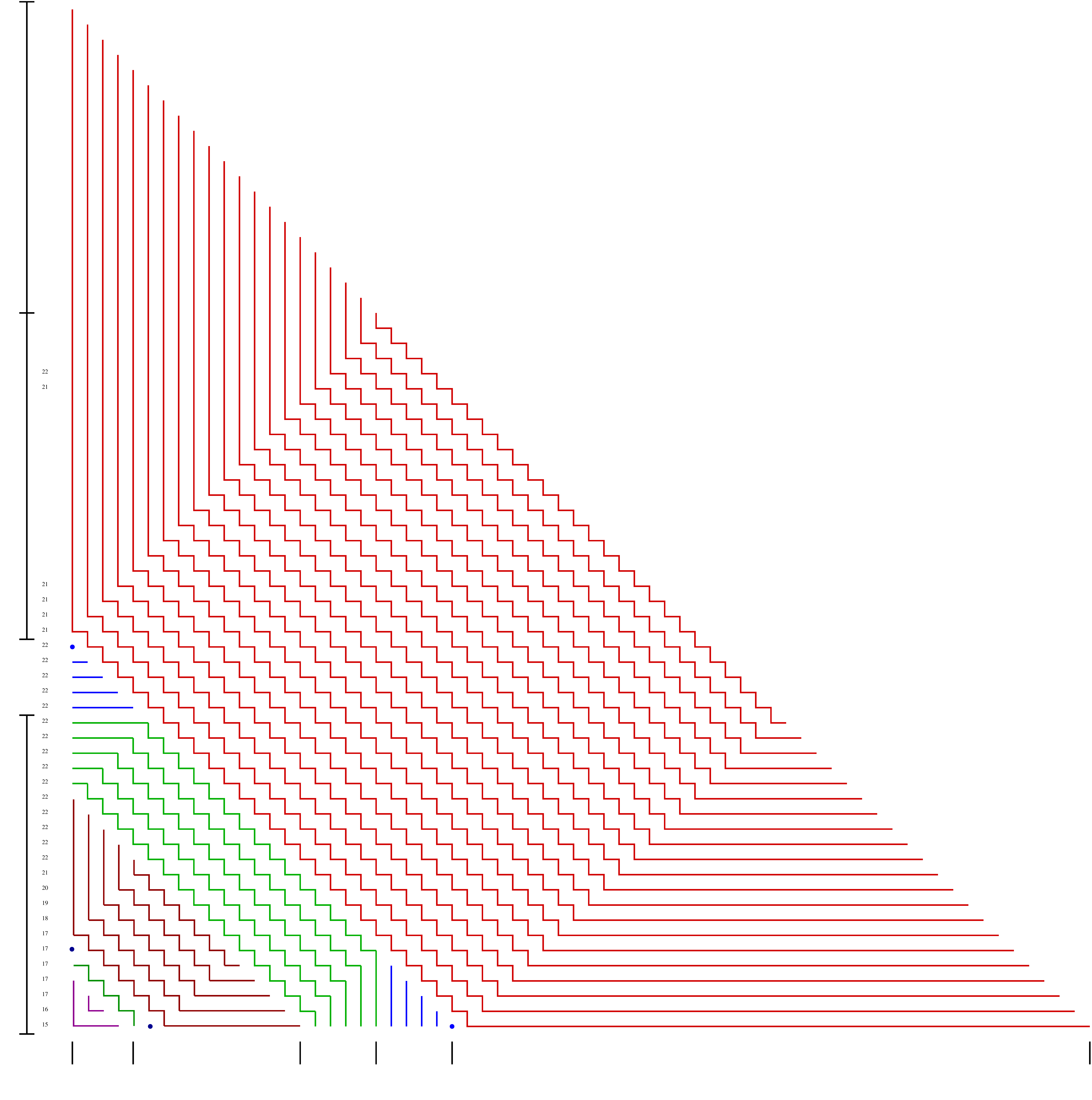}
    \caption{The table $T_{68}$ with the 22-local queue assignment obtained with $k=21$.
        Chains in the families $\calA$, $\calB$, and $\calC$
        are colored red, blue and green respectively.
        The chains given by a 6-local queue assignment of  $T_{16}$
        are darker, the $T_4$ of the second recursive call is magenta.}
    \label{fig:schema68}
\end{figure}

  For $n$ small ($n < 32$) it may be that $n < 3k$.  In this case the families
  $\calB$ and~$\calC$ are empty and $\calF_{n'}$ with $n' = n -2k$ can be
  chosen as a family of nested elbows, see Figure~\ref{fig:q13}. The
  corresponding queue assignment is even $k$-local. From now on we assume that
  $3k \leq n$.

\begin{figure}
    \centering
    \includegraphics[width = 10em]{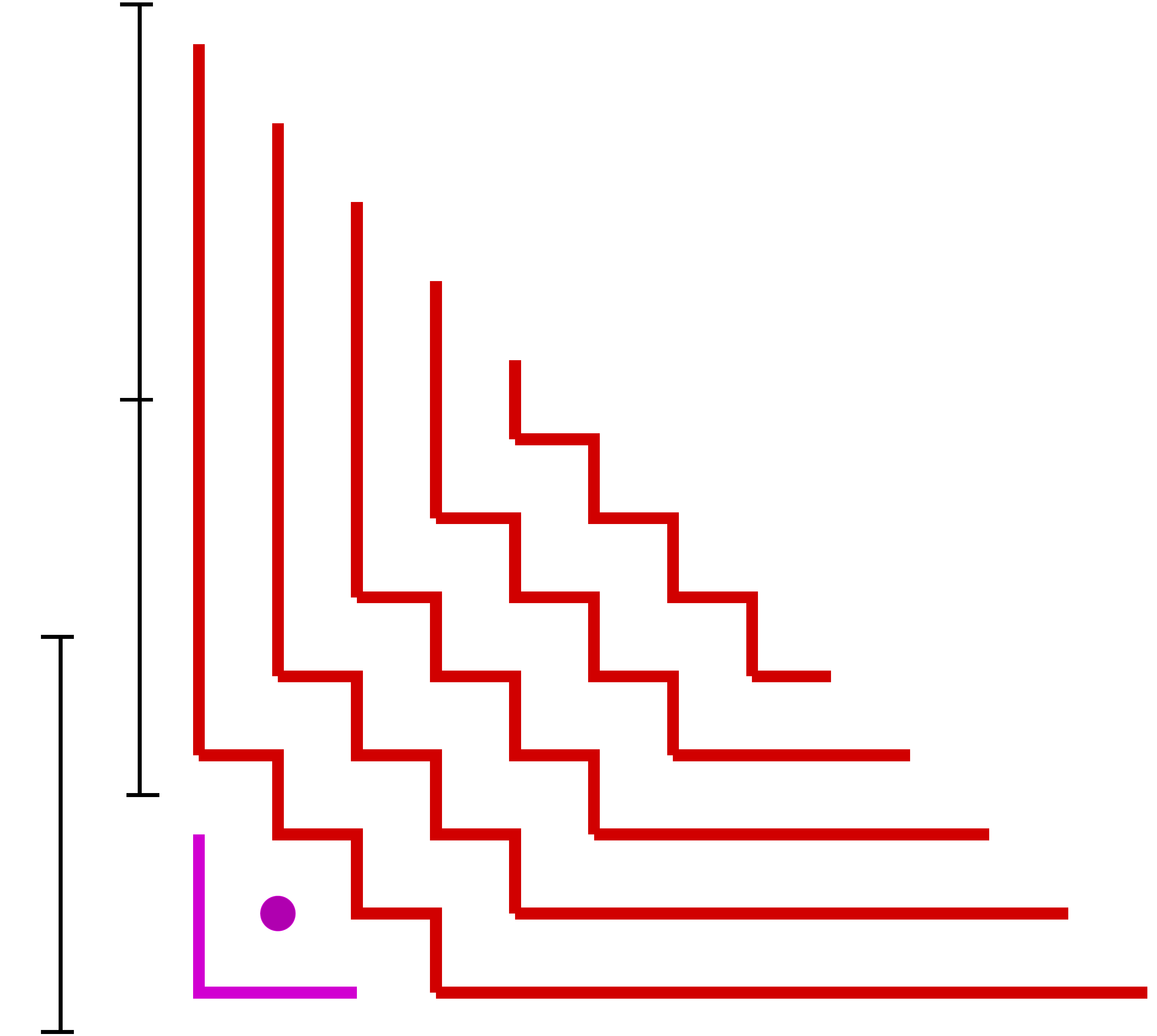}
    \caption{
        The table $T_{13}$ with the 6-local queue assignment for $K_{14}$ obtained
        with $k=5$. Since $n < 3k$ families $\calB$, and $\calC$ are empty. The
        queue assignment of $T_{3}$ in the lower corner consists of two nested elbows.}
    \label{fig:q13}
\end{figure}

Family $\calB$ consists of 
$n-3k$ chains $B_1,\ldots,B_{n-3k}$ where $B_\beta$ is composed of
two disjoint blocks. The first block contains the first $\beta$ elements
in row $n+1-2k-\beta$ and second block the $\beta$ first elements in column
$n-2k-\beta$.  Formally, for $\beta = 1,\ldots,n-3k$ we have
 \begin{align*} B_\beta = &\,\{ (\,x\,,n+1-2k-\beta) \in T_n \mid x \leq \beta \}\\
 \cup &\,\{ (n+1-2k-\beta,\,y\,) \in T_n \mid y \leq \beta \}.
 \end{align*}

 Family $\calC$ also consist of $n-3k$ chains $C_1,\ldots,C_{n-3k}$.  Chain
 $C_\gamma$ is composed of three blocks. The first block consists of the
 $n-3k+2-\gamma$ first elements in row $k+1-\gamma$, the second block starts
 with the last element of the first block and continues alternatingly down and
 right for $8k-2n-2$ steps until column $k+1-\gamma$, the last block consists
 of $n-3k+2-\gamma$ lowest elements in row $k+1-\gamma$.  Formally, for
 $\gamma = 1,\ldots,n-3k$ we set
 \begin{align*} 
    C_\gamma = &\,\{ (\,x\,,k+1-\gamma) \in T_n \mid  x \leq n-3k+2-\gamma \}\\ 
    \cup &\, \{ 
    \begin{aligned}[t]
        (\,x,y\,) \in T_n \mid 
        &\ y-x \leq 4k-n-1 \text{ and } x-y \leq 4k-n-1  \text{ and }\\
        &\ n-2k+2-2\gamma \leq x+y \leq  n-2k+3-2\gamma \}
    \end{aligned}\\
    \cup &\,\{ (k+1-\gamma,\,y\,) \in T_n \leq y \leq n-3k+2-\gamma\}.
 \end{align*}
 
 Observe that the subset of $T_n$ of elements not covered by any of
 the chains in $\calA \cup \calB \cup \calC$ is exactly $T_{n'}$ with
 \[
 n' = n-2k-2(n-3k) = 4k-n.
 \]
 We call induction on $T_{n'}$.
 We claim that $k' = \lceil (1-1/\sqrt{2})(n'+1)\rceil$
 is at most $7k-2n -1$, i.e., we have to show
 $(1-1/\sqrt{2})(n'+1) \leq 7k-2n -1$:
 \begin{align*}
   &       &  (1-1/\sqrt{2})(4k-n+1) &\leq 7k-2n -1 \\
   &\iff   &  (1+1/\sqrt{2})n + (2-1/\sqrt{2}) &\leq  (3+2\sqrt{2})k \\
   &\kern4pt\Longleftarrow &  (1+1/\sqrt{2})n + (2-1/\sqrt{2}) &\leq
                      (3+2\sqrt{2})(1-1/\sqrt{2})(n+1)\\
   &\iff   &  (1+1/\sqrt{2})n + (2-1/\sqrt{2}) &\leq (1+1/\sqrt{2})n + (1+1/\sqrt{2})\\
   &\iff   &  1 &\leq \sqrt{2}.
 \end{align*}
 So our family $\calF_n$ of weakly monotonously decreasing chains covering
$T_n$ is given by $\calF_n = \calA \cup \calB \cup \calC \cup \calF_{n'}$,
where $\calF_{n'}$ is the family obtained by induction on $T_{n'}$.
 
 It remains to show that every hook of $T_n$ is hit by at most $k + 1$ chains
of the family~$\calF_n$.  Indeed, for hook $H_i$ consisting of column $i$ and row $n+2-i$
we have the following.
 \begin{itemize}
  \item For $i \leq n-3k$ hook $H_i$ is hit by $A_1,\ldots,A_i$,
$B_i,\ldots,B_{n-3k}$, $C_1,\ldots,C_{n-3k}$, and at most $k'+1$ chains in
$\calF_{n'}$.  This makes in total no more than $i + (n-3k-(i-1)) + (n-3k) + (k'+1) \leq
2n-6k +1 + (7k-2n) = k+1$ chains, as desired.
  \item For $n-3k < i \leq n'$ hook $H_i$ is hit by $A_1,\ldots,A_i$,
$C_1,\ldots,C_{n-3k}$, and at most $n'+1-i$ chains in $\calF_{n'}$.  This
makes in total no more than $i + (n-3k) + (n'+1-i) = n-3k + 1 + n' =
n-3k + 1 +(4k-n) = k+1$ chains, as desired.
  \item For $n' < i \leq k$ hook $H_i$ is hit by $A_1,\ldots,A_i$ and
$C_1, \ldots,C_{k-i+1}$, which makes a total of $k+1$ chains.
  \item For $k < i \leq n/2$ hook $H_i$ is hit by $A_1,\ldots,A_k$ and at most
one chain in $\calB$, thus in total by at most $k+1$ chains.
  \item Finally, for $n/2 < i$ the argument is symmetric by reflecting $H_i$
  along the main diagonal.
\end{itemize}
This concludes the proof.
\end{proof}

\section{Local and Union Page Numbers}
\label{sec:page-number}

For book embeddings, it is convenient to think of the spine as being
circularly closed.  The placement of the vertices together with straight-line
edges yields a convex drawing of $K_n$. A page assignment is a partition of
the edges into non-crossing subsets, i.e., into outerplanar subdrawings of
this drawing of $K_n$.

First, we analyze the outerplanar subgraphs on each page of a book embedding and thereby show the lower bound of \cref{thm:page-number}.
The proof gives insight into how book embeddings for a matching upper bound on the local or union page number should look like.
This bound is also the best lower bound we obtain for the union page number.
We then prove upper bounds both on the local and union page number.

\begin{lemma}\label{lem:page-LB} 
   For any $n$ we have
  $\displaystyle \lpn(K_n) > \frac{1}{3}n - 1$.
\end{lemma}
\begin{proof}
  Let $\calP$ be a page assignment of $K_n$ which minimizes the
  local page number. We assume that $\lpn(\calP)\leq n/3$,
  otherwise we are done.
  Let $k$ be the average number of vertex-page
  incidences over all vertices, i.e.,
  $k = \frac{1}{n}\sum_{P \in \calP} |V_P|$.  We shall show that
  $k > \frac{1}{3}n - 1$, which in particular proves that
  $\lpn(\calP) > n/3 - 1$. Later we will
  use that $k \leq \lpn(\calP)\leq n/3$, i.e.,
  \begin{equation}\label{eq:k-leqn/3}
    k \leq \frac{n}{3}.
  \end{equation}
 
 Note that every edge of $K_n$ belongs to exactly one page of
 $\calP$. Now for each page $P \in \calP$ we consider an outerplanar
 graph $O_P$ consisting of all edges of $P$ and their incident vertices $V_P$
 together with the edges of the convex hull~$C_P$ of $V_P$.
 For each page $P\in \calP$ we color the edges of $O_P$:
 \begin{itemize}
  \item The \emph{black edges} are edges of $C_P$ belonging to $P$.
  \item The \emph{red edges} are edges of $C_P$ which do not belong to $P$.
  \item The \emph{green edges} are inner edges of $O_P$ which belong to $P$.
  \end{itemize}
Observe that every edge $e$ of $K_n$ has a color in $\{ \text{black},
\text{green}\}$ for exactly one page, while $e$ may be red for any number of
pages.

For a vertex $v$ and a page $P$ containing $v$, let the \emph{forward edge} $\fwdP(v)$ at
$v$ be the edge of $C_P$ which leaves $v$ in clockwise direction. Let $r_v$
be the number of pages for which the forward edge of $v$ is red, and $b_v$ be the
number of pages for which the forward edge of $v$ is black.  As $v$ has exactly one
forward edge on each page,~$v$ is incident to exactly $r_v + b_v$ pages in
$\calP$.  Hence, denoting $R = \sum_v r_v$ and $B = \sum_v b_v$, we have
\begin{equation}\label{eq:kn=R+B}
  k\,n = \sum_{P \in \calP} |V_P| = \sum_v (r_v+b_v) = R + B. 
\end{equation}

Now for each page $P$ and each edge $e = uv$ of $C_P$ with $u$
clockwise followed by $v$, let $\len(e)$ be the distance along 
$K_n$ when going clockwise from $u$ to $v$, i.e.,
if $u=v_i$ and $v=v_j$ and $i < j$ then
$\len(e) = j-i$ and $\len(e) = j-i+n$ if $j<i$.
Since $C_P$ is a cycle, we have $\sum_{e\in C_P} \len(e) = n$.  Thus

\begin{align*}
  |\calP|&\cdot n   = \sum_{P \in \calP} \Big(\sum_{e \in C_P} \len(e)\Big)
                    = \sum_{P \in \calP} \Big(\sum_{v \in V_P} \len(\fwdP(v))\Big)\\
                  & = \sum_v \Big(\sum_{P:v\in V_P} \len(\fwdP(v))\Big)
    \overset{(\diamond)}{\geq} \sum_v \Big( \sum_{\ell = 1}^{b_v} \ell\Big)
                    \geq \sum_v \frac{b_v^2}{2}
   \overset{(\ast)}{\geq} \frac{1}{2n} \Big(\sum_v b_v\Big)^2\\
                  & = \frac{1}{2n} B^2
   \overset{\eqref{eq:kn=R+B}}{=} \frac{1}{2n}(kn-R)^2
                    \geq \frac{1}{2n}(k^2n^2 - 2knR)
                    = \frac{k^2}{2}n - kR
   \overset{\eqref{eq:k-leqn/3}}{\geq} \frac{k^2}{2}n - \frac{n}{3}R.
\end{align*}
For $(\diamond)$ ignore red forward edges at $v$ and use that the black
forward edges at~$v$ are pairwise distinct and for $(\ast)$ use the
Cauchy-Schwarz inequality with the vectors $(b_{v_1},\ldots,b_{v_n})$ and
$(1,\ldots,1)$.

Dividing both sides of the above by $n$ we get
 \begin{equation} |\calP| \geq
\frac{k^2}{2}-\frac{R}{3}.\label{eq:num-pages-lower-bound}
 \end{equation}
 
Now consider the green edges in $K_n$.
Since $O_P$ is outerplanar and the green edges of $O_P$ are the inner edges
there are at most $|V_P|-3$ green edges on page~$P$.
Therefore, we have
\begin{align}
    \label{eq:green-edges-upper-bound}
    \#\text{green edges}
    &\leq  \sum_{P \in \calP} (|V(P)|-3)
    = kn - 3|\calP|
    \overset{\eqref{eq:num-pages-lower-bound}}{\leq} kn - \frac{3k^2}{2} + R \text{ and}\\
    \label{eq:green-edges-count}
    \#\text{green edges} 
    &= |E(K_n)| - \#\text{black edges}
    = \binom{n}{2} - B
    \overset{\eqref{eq:kn=R+B}}{=} \binom{n}{2} - kn + R.
\end{align}
Combining \eqref{eq:green-edges-upper-bound} and
\eqref{eq:green-edges-count} we conclude:
\begin{align*}
    && kn - \frac{3k^2}{2}+R 
    &\geq \binom{n}{2} - kn + R \\
    \iff&& 
    0 &\geq \frac{3k^2}{2} - 2kn + \binom{n}{2} \\
    \iff&& 
    0 &\geq k^2 -\frac{4n}{3}k + \frac{n(n-1)}{3}\\
    \implies&&
    k &\geq \frac{2n}{3} - \sqrt{\left(\frac{2n}{3}\right)^2 - \frac{n(n-1)}{3}} \\
    \implies&&
    k &\geq 
    \frac{2n}{3} - \sqrt{\frac{n^2}{9} + \frac{n}{3}} 
    > \frac{2n}{3} - \sqrt{\frac{(n+3)^2}{9}} = \frac{n}{3} - 1
\end{align*}
Thus we have $\displaystyle k > \frac{n}{3} - 1$, as desired.
\end{proof}

Note that for the lower bound to be tight there can be no red edges,
i.e., each page contains the edges of the convex hull of its vertices.
This is the case in the construction for the upper bound on the local page number for \cref{thm:page-number}, which is given next.

\begin{lemma}\label{lem:local-page-UB}
  For any $n$, we have\quad
  $\displaystyle \lpn(K_n) \leq \frac{1}{3}n + 4$.
\end{lemma}
\begin{proof}
 We shall show that if $n = 18k-3$ for some positive integer $k$,
 then $\lpn(K_n) \leq n/3=6k-1$.
 For $n$ of the form $n = 18k-3+i$ with $i < 18$ we get a page assignment
 with locality $6k-1+i$ by adding the stars
 of the $i$ additional vertices on an extra page each. A second
 option is to use a page assignment of $K_{18(k+1)-3}$ and
 remove $18-i$ vertices, this yields a page assignment with
 locality $6(k+1) - 1 = 6k+5$. By taking the better of these two choices we
 achieve a locality of at most $n/3+4$, as desired.
 
 From now on we
 assume that $n = 18k-3$, i.e., $k = (n+3)/18$. We define the length of an
 edge as the shorter distance between its two endpoints along the cyclic
 ordering. The length of edge $e$ is denoted $\len(e)$.
 As $n = 18k-3$ is odd, there are exactly $(n-1)/2 = 9k-2$
 different lengths, each realized by exactly $n$ edges. For each vertex $v$
 of $K_n$ we define a set of $k$ pages, each containing~$v$, and
 together covering exactly one edge of each length $1,\ldots,9k-2$.
 These pages each contain an outerplanar graph and are denoted by $O_v(t)$, where $t=0,\ldots,k-1$.  
 The page $O_v(0)$ contains seven edges (and five vertices) while for $t>0$ page $O_v(t)$ has nine edges (and six vertices). In total this makes the needed
 $9(k-1)+7 = 9k-2 = (n-1)/2$ edge lengths.
 
 Recall that the vertices of $K_n$ are $v_1,\ldots,v_n$ in this cyclic ordering.
 Below we describe the pages corresponding to $v_1 = v_{n+1}$. For ease of notation,
 let $O(t) = O_{v_1}(t)$. For $t = 0,\ldots,k-1$, we define the vertices $r_1(t),\ldots,r_6(t)$
 of $O(t)$:
  \begin{align*}
   r_1(t) &= v_1 = v_{18k-2} & r_2(t) &= v_{2k-2t} & r_3(t) &= v_{5k+1} \\ 
   r_4(t) &= v_{8k-t} & r_5(t) &= v_{8k+t} & r_6(t) &= v_{13k+2t}
  \end{align*}

  \begin{figure} 
    \centering 
    \includegraphics{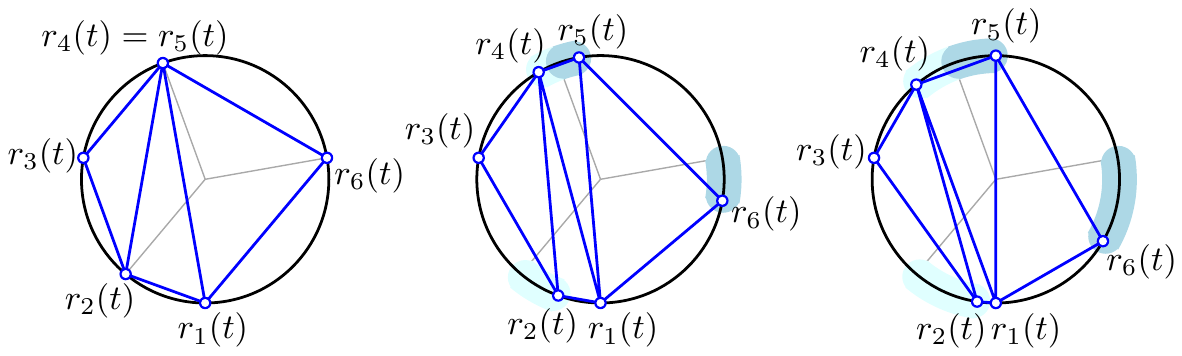}
    \caption{$O(t)$ for $t=0$ (left), $t \approx k/2$ (middle) and $t = k-1$ (right)}
  \label{fig:local-page-UB}
 \end{figure}
 We refer to \cref{fig:local-page-UB} for an illustration.  Note that for 
 $t=0$ we have $r_4(t) = r_5(t)$ and for all $t \leq k$ the vertices $r_1(t),\ldots,r_6(t)$
 appear in the order of their indices in the  cyclic ordering of $K_n$.  The edges of
 $O(t)$ are the cycle edges $e_{12}(t)=r_1(t)r_2(t)$, $e_{23}(t)=r_2(t)r_3(t)$,
 $e_{34}(t)=r_3(t)r_4(t)$, $e_{45}(t)=r_4(t)r_5(t)$, $e_{56}(t)=r_5(t)r_6(t)$, $e_{61}(t)=r_6(t)r_1(t)$,
 except for $e_{45}(0)$ which would be a loop, and the inner edges 
 $e_{14}(t)=r_1(t)r_4(t)$, $e_{24}(t)=r_2(t)r_4(t)$, $e_{15}(t)=r_1(t)r_5(t)$,
 again for $t=0$ there is an exception, we disregard $e_{15}(0)$ because it equals
 $e_{14}(0)$. Note that $O(t)$ is indeed outerplanar.
 We claim that for every length $\ell$ in the interval $[1,9k-2]$ there is an
 edge of length $\ell$ in some $O(t)$.
 
 \begingroup
 \allowdisplaybreaks
 \begin{align*}
  \ell \text{ odd, }&\ell\in [1,2k-1] &:&\quad  \len(e_{12}(t)) = (2k-2t) - 1 = 2k-2t-1 \\
  \ell \text{ even, }&\ell\in [1,2k-1] &:&\quad  \len(e_{45}(t)) = (8k+t) - (8k-t) = 2t\\
                    &\ell\in [2k,3k-1]&:&\quad  \len(e_{34}(t)) = (8k-t) - (5k+1) = 3k-t-1\\
  \ell \text{ odd, }&\ell\in [3k,5k-1]&:&\quad  \len(e_{23}(t)) = (5k+1) - (2k-2t) = 3k+2t+1\\
  \ell \text{ even, }&\ell\in [3k,5k-1]&:&\quad  \len(e_{61}(t)) = (18k-2) - (13k+2t) = 5k-2t-2\\
                    &\ell\in [5k,6k-1]&:&\quad  \len(e_{56}(t)) =  (13k+2t) - (8k+t) = 5k+t\\
                    &\ell\in [6k,7k-1]&:&\quad  \len(e_{24}(t)) =   (8k-t) - (2k-2t) = 6k+t\\
                    &\ell\in [7k,8k-1]&:&\quad  \len(e_{14}(t)) = (8k-t) - 1 = 8k-t-1\\
                    &\ell\in [8k,9k-2]&:&\quad  \len(e_{15}(t)) = (8k+t) - 1 = 8k+t-1
 \end{align*}%
 \endgroup%
For a vertex $v_i$ and some $t$ we obtain the page $O_{v_i}(t)$ from $O(t)$ by
a rotation which maps $v_1$ to $v_i$. Hence, for each $v_i$ we get a
collection $\calP_i = \{O_{v_i}(t) \mid t = 0,\ldots,k-1\}$ of pages. We claim
that $\calP = \bigcup_{v} \calP_v$ covers all the edges of~$K_n$. Consider an
arbitrary edge $v_av_b$, we assume that the arc from $v_a$ to $v_b$ is the
shorter arc, i.e., the length of the arc is $\len(v_av_b)=\ell$. From the
analysis above we know that there is a unique $t$ and a unique edge
$e_{ij}(t)\in O(t)$ with $\len(e_{ij}(t))=\ell$. There is a rotation which maps
$r_i(t)$ to $v_a$ and consequently also~$r_j(t)$ to $v_b$.  If this rotation
maps $v_1=r_1(t)$ to $v_c$, then $O_{v_c}(t)$ contains the edge $v_av_b$.
Hence, $\calP$ is a covering of the edges of $K_n$ with outerplanar graphs.

In $\calP_1$ there are $6(k-1)+5 = 6k-1$ vertex-page incidences, hence
the total number of vertex-page incidences in $\calP$ is
$n(6k-1) = n^2/3$.  Due to symmetry, each of
the $n$ vertices is incident to exactly $n/3$ pages. This proves that $\lpn(K_n)
\leq n/3$ whenever $n$ is of the form $n=18k-3$ for some positive integer~$k$.
\end{proof}

As the union page number is lower-bounded by the local page number and upper-bounded by the (global) page number, we immediately have $ n/3 - \mathcal{O}(1) = \lpn(K_n) \leq \upn(K_n) \leq \pn(K_n) = \lceil n/2 \rceil $.
We improve on the upper bound by constructing a book embedding consisting of $ 4/9 + \mathcal{O}(1) $ union pages, which concludes the proof of \cref{thm:page-number}.

\begin{lemma}\label{lem:union-page-UB}
 The union page number of $ K_n $ satisfies
 \[ \upn(K_n) \leq \frac{4}{9} n + 18. \] 
\end{lemma}

\begin{proof}
    For any $ k > 0 $ that is divisible by $ 3 $ and $ n = 18k $, we prove that $ \upn(K_n) \leq 4n/9 + 4 $.
    For any other $ n $, we add stars or use the book embedding for $ K_{n'} $, where $ n' $ is the smallest integer with $ n' \geq n $ and $ n' = 54m $ for some integer $ m $.
    The better of these options gives $ \upn(K_n) \leq 4n/9 + 18 $.
    
    The vertices of $ K_n $ are denoted by $ v_0, \dots, v_{n - 1} $ and lie on the circularly closed spine in this ordering.
    All indices are taken modulo $ n $.
    Let the \emph{length} $ \len(vw) $ of an edge $ vw $ denote the shorter distance between $ v $ and $ w $ along the cyclic vertex ordering.
    We have $ n/2 = 9k $ different lengths, where length $ 9k $ is realized by $ n/2 $ edges and all other lengths by $ n $ edges.
    We first define $ n/3 $ union pages $ \calP_1, \dots, \calP_{n/3} $ that cover $ 7/9 $ of all possible lengths and then cover the remaining edges with unions of stars.
    The union pages $ \calP_1, \dots, \calP_{n/3} $ consist of graphs $ G(i, t) $ and $ H(i, t) $ defined below.
    We refer to \cref{fig:union-page-UB} for an illustration.
    For $ t = 0, \dots, k - 1 $, we define vertices
    \begin{align*}
        r_1(t) &= v_{1 + t} &
        r_2(t) &= v_{8k + 1 - t} & 
        r_3(t) &= v_{9k + 1 + 2t} &
        r_4(t) &= v_{12k - t} \\
        r_5(t) &= v_{17k - 2t} &
        s_1(t) &= v_{3k} & 
        s_2(t) &= v_{8k + 1 + t}.
    \end{align*}
    For each $ t = 0, \dots, k - 1 $, the graph $ G(1, t) $ is then defined to consist of the vertices $ r_1(t), \dots, r_5(t) $ and the edges 
    $e_{12}(t)=r_1(t)r_2(t)$, 
    $e_{13}(t)=r_1(t)r_3(t)$,
    $e_{14}(t)=r_1(t)r_4(t)$,
    $e_{15}(t)=r_1(t)r_5(t)$,
    $e_{23}(t)=r_2(t)r_3(t)$,
    $e_{34}(t)=r_3(t)r_4(t)$, and
    $e_{45}(t)=r_4(t)r_5(t)$.
    The graph $ H(1, t) $ consists of a single edge $ s_1(t)s_2(t) $ for $ t = 0, \dots, k - 1 $.
    For $ i = 1, \dots, n $, we obtain $ G(i, t) $ by rotating $ G(1, t) $, i.e., we use the vertices $ r'_h = v_j + i - 1 $ instead of $ r_h = v_j $ for $ h = 1, \dots, 5 $.
    Similarly, $ H(i, t) $ is obtained from $ H(1, t) $ by rotation.
    
    \begin{figure}
        \centering\includegraphics[page = 3]{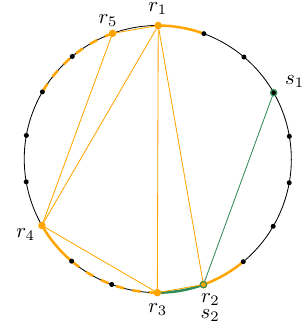}
        \caption{The graphs $ G(1, 0) $ (orange) and $ H(1, 0) $ (green). The orange, respectively green, arcs indicate the movement of the vertices for $ t = 0, \dots, k - 1 $.}
        \label{fig:union-page-UB}
    \end{figure}
    
    We claim that for $ i = 1, \dots, n $ and $ t = 0, \dots, k - 1 $, the graphs $ G(i, t) $ and $ H(i, t) $ cover all edges of length $ \ell $ with $ \ell = k, \dots, 3k $ and $ \ell = 4k + 1, \dots, 9k $.
    We only consider the graphs for $ i = 1 $ and observe that each length is covered at least once.
    By symmetry, all edges of the claimed lengths are covered.
    Recall that $ k $ is divisible by $ 3 $.
    An edge of length $ \ell $ can be found in $ G(1, t) $ or $ H(1, t) $ as follows:
    \begin{equation*}
        \begin{array}{rlclcl}
            \ell \equiv 0 \mod 3, & \ell \in [k, 3k] & : & \len(e_{23}(t)) &=& k + 3t \\
            \ell \equiv 1 \mod 3, & \ell \in [k, 3k] & : & \len(e_{15}(t)) &=& k + 1 + 3t \\
            \ell \equiv 2 \mod 3, & \ell \in [k, 3k] & : & \len(e_{34}(t)) &=& 3k - 1 - 3t \\
            & \ell \in [4k + 1, 5k] & : & \len(e_{45}(t)) &=& 5k - t \\
            & \ell \in [5k + 1, 6k] & : & \len(s_1(t)s_2(t)) &=& 5k + 1 + t \\
            \ell \text{ even}, & \ell \in [6k + 1, 8k] & : & \len(e_{12}(t)) &=& 8k - 2t \\ 
            \ell \text{ odd}, & \ell \in [6k + 1, 8k] & : & \len(e_{14}(t)) &=& 6k + 1 + 2t \\
            & \ell \in [8k + 1, 9k] & : & \len(e_{13}(t)) &=& 9k - t
        \end{array}
    \end{equation*}
    
    We now use the given graphs to define the union pages $ \calP_1, \dots, \calP_{n/3} $.
    For $ i = 1, \dots, n/3 $, we define $ \calP_i $ as the union of $ G(i, t) $, $ G(i + n/3, t) $, and $ G(i + 2n/3, t) $ for $ t = 0, \dots, k - 1 $ and $ H(i, t)$, $ H(i + n/3, t) $, and $ H(i + 2n/3, t) $ for $ t = 1, \dots, k - 1 $. 
    To observe that $ \calP_1, \dots, \calP_{n/3} $ are indeed union pages, we list for each vertex $ v_1, \dots, v_n = v_0 $ by which vertices of $ \calP_1 $ it is hit.
    For ease of presentation, we assume that $ k $ is even.
    For odd $ k $, swap \emph{odd} and \emph{even} in the listing below.
    \begin{equation*}
        \begin{array}{lllcll}
            v_1, \dots, v_{k}, & v_{6k + 1}, \dots, v_{7k}, & v_{12k + 1}, \dots, v_{13k} & : & r_1(t) \\
            v_{k + 2}, \dots, v_{2k + 1}, & v_{7k + 2}, \dots, v_{8k + 1}, & v_{13k + 2}, \dots, v_{14k + 1} & : & r_2(t)\\
            v_{2k + 2}, \dots, v_{3k}, & v_{8k + 2}, \dots, v_{9k}, & v_{14k + 2}, \dots, v_{15k} & : & s_2(t) & t \geq 1 \\
            v_{3k}, & v_{9k}, & v_{15k} & : & s_1(t) & t \geq 1 \\
            v_{3k + 1}, \dots, v_{5k}, & v_{9k + 1}, \dots, v_{11k}, & v_{15k + 1}, \dots, v_{17k} & : & r_3(t) & \text{odd indices} \\
            v_{3k + 1}, \dots, v_{5k}, & v_{9k + 1}, \dots, v_{11k}, & v_{15k + 1}, \dots, v_{17k} & : & r_5(t) & \text{even indices} \\
            v_{5k + 1}, \dots, v_{6k}, & v_{11k + 1}, \dots, v_{12k}, & v_{17k + 1}, \dots, v_{18k} & : & r_4(t) \\
        \end{array}
    \end{equation*}
    The vertices $ v_{3k} $, $ v_{9k} $, and $ v_{15k} $ are hit by $ s_1(t) $ for all $ t $ and by $ s_2(k - 1) $, whereas all other vertices are hit for at most one $ t $.
    In particular, each vertex that is contained in some $ G(j, t) $ is not contained in any other component of $ \calP_1 $.
    In contrast, the union of the graphs $ H(j, t) $, $ j = i, i + n/3, i + 2n/3 $ and $ t = 1, \dots, k - 1 $, forms a single connected component whose edges do not cross.
    Hence, each connected component of $ \calP_1 $ is crossing-free and by symmetry $ \calP_1, \dots, \calP_{n/3} $ are union pages.
    None of the graphs $ H(i, 0) $, for $ i = 1, \dots, n $, is contained in the pages $ \calP_1, \dots, \calP_{n/3} $, i.e., the edges of length $ 5k + 1 $ are left to cover.
    We cover these edges with two additional union pages, each containing a perfect matching.
    
    Finally, we define union pages consisting of disjoint unions of stars to cover the remaining edge lengths $ 1, \dots, k - 1 $ and $ 3k + 1, \dots, 4k $.
    For this, we define stars $ S_i $ consisting of the edges $ v_{i} v_{i + 1}, \dots, v_{i} v_{i + k - 1} $ and $ T_i $ consisting of the edges $ v_{i} v_{i + 3k + 1}, \dots, v_{i} v_{i + 4k - 1} $ for $ i = 1, \dots, n $.
    For $ i = 1, \dots, k $, the union page $ \mathcal{S}_i $, respectively $ \mathcal{T}_i $, is defined as the union of $ S_{i + jk} $, respectively $ T_{i + jk} $, where $ j = 0, \dots, 17 $.
    As each union page is the disjoint union of stars, each connected component is crossing-free.
    The union pages $ \mathcal{S}_1, \dots, \mathcal{S}_k $ and $ \mathcal{T}_1, \dots, \mathcal{T}_k $ cover all remaining edge lengths except for the length $ 4k $.
    Again, we use two additional union pages containing a perfect matching each.
    Summing up, we have $ n/3 + 2 + 2k + 2 = 4n/9 + 4 $ union pages.
\end{proof}

Comparing the presented construction with the lower bound of \cref{lem:page-LB}, we remark that we have $ n^2/18 + \Theta(n) $ connected components and $ \Theta(n^2) $ red edges due to the stars.
To obtain an upper bound of $ n/3 $, however, we need exactly $ n^2/18 $ connected components that are partitioned into $ n/3 $ union pages.
In this case, each union page uses all vertices and each connected component is a maximal outerplanar graph, i.e., there are no red edges.
It remains open whether such a book embedding exists.
Note that the book embedding constructed for \cref{lem:local-page-UB} consists of $ n^2/18 $ pages and has no red edges.
Partitioning these pages into union pages containing $ n/6 $ outerplanar graphs each thus would suffice to prove an upper bound of $ n/3 + \mathcal{O}(1) $ for the union page number of $ K_n $.

\section{Conclusions}

We have shown bounds on the local page number, the local queue number, and the union queue number of complete graphs that are tight up to a constant additive term.
However, there remains a gap between the lower bound of $ n/3 - \mathcal{O}(1) $ and the upper bound of $ 4n/9 + \mathcal{O}(1) $ on the union page number of $ K_n $.
\begin{question}
    What is the union page number of complete graphs?
\end{question}

Comparing queues and stacks, we find that both in the local and in the union setting, queues are more powerful than stacks for representing complete graphs as both the local and the union queue number is smaller than the respective variant of the page number.

Finally, we point out complete bipartite graphs as another dense graph class.
Heath and Rosenberg~\cite{hr-92} proved $ \qn(K_{m, n}) = \lceil m / 2 \rceil $, where $ m \leq n $.
For the page number, it is known that
$\pn(K_{m, n}) = m $ if $ n \geq m^2 - m + 1$~\cite{bk-79},
$\pn(K_{n, n}) \leq \lfloor 2n / 3 \rfloor + 1$,
$\pn(K_{\lfloor n^2/4 \rfloor, n}) \leq n - 1$~\cite{eno-97}, and in general
$\pn(K_{m, n}) \leq \lceil (m + 2n) / 4 \rceil$~\cite{mww-88}.
In light of the unclear situation for the page number, we ask for the local and union variants of queue number and page number of complete bipartite graphs.

\section*{Acknowledgments}
The first, third and fourth author would like to thank the organizers and all participants of the Seventh Annual Workshop on Geometry and Graphs in Barbados, where part of this research was carried out.

\bibliographystyle{splncs04}
\bibliography{references-gd}

\begin{thebibliography}{10}
\providecommand{\url}[1]{\texttt{#1}}
\providecommand{\urlprefix}{URL }
\providecommand{\doi}[1]{https://doi.org/#1}

\bibitem{aa-89}
Algor, I., Alon, N.: The star arboricity of graphs. Discrete Mathematics
  \textbf{75}(1),  11--22 (1989). \doi{10.1016/0012-365X(89)90073-3}

\bibitem{bk-79}
Bernhart, F., Kainen, P.C.: The book thickness of a graph. Journal of
  Combinatorial Theory, Series B  \textbf{27}(3),  320--331 (1979).
  \doi{10.1016/0095-8956(79)90021-2}

\bibitem{bsu-18}
Bläsius, T., Stumpf, P., Ueckerdt, T.: Local and union boxicity. Discrete
  Mathematics  \textbf{341}(5),  1307--1315 (2018).
  \doi{10.1016/j.disc.2018.02.003}

\bibitem{dfgklnu-21}
Damásdi, G., Felsner, S., Gir{\~a}o, A., Keszegh, B., Lewis, D., Nagy, D.T.,
  Ueckerdt, T.: On covering numbers, young diagrams, and the local dimension of
  posets. SIAM Journal on Discrete Mathematics  \textbf{35}(2),  915--927
  (2021). \doi{10.1137/20M1313684}

\bibitem{dehmw-20}
Dujmović, V., Eppstein, D., Hickingbotham, R., Morin, P., Wood, D.R.:
  Stack-number is not bounded by queue-number (2020),
  \url{https://arxiv.org/abs/2011.04195}

\bibitem{eno-97}
Enomoto, H., Nakamigawa, T., Ota, K.: On the pagenumber of complete bipartite
  graphs. Journal of Combinatorial Theory, Series B  \textbf{71}(1),  111--120
  (1997). \doi{10.1006/jctb.1997.1773}

\bibitem{el-20}
Esperet, L., Lichev, L.: Local boxicity (2020),
  \url{https://arxiv.org/abs/2012.04569}

\bibitem{hlr-92}
Heath, L.S., Leighton, F., Rosenberg, A.: Comparing queues and stacks as
  machines for laying out graphs. SIAM Journal on Discrete Mathematics
  \textbf{5}(3),  398--412 (1992). \doi{10.1137/0405031}

\bibitem{hr-92}
Heath, L.S., Rosenberg, A.: Laying out graphs using queues. SIAM Journal on
  Computing  \textbf{21}(5),  927--958 (1992). \doi{10.1137/0221055}

\bibitem{kmmssuw-20}
Kim, J., Martin, R.R., Masařík, T., Shull, W., Smith, H.C., Uzzell, A., Wang,
  Z.: On difference graphs and the local dimension of posets. European Journal
  of Combinatorics  \textbf{86},  103074 (2020).
  \doi{10.1016/j.ejc.2019.103074}

\bibitem{ku-16}
Knauer, K., Ueckerdt, T.: Three ways to cover a graph. Discrete Mathematics
  \textbf{339}(2),  745--758 (2016). \doi{10.1016/j.disc.2015.10.023}

\bibitem{mm-21}
Majumder, A., Mathew, R.: Local boxicity and maximum degree (2021),
  \url{https://arxiv.org/abs/1810.02963}

\bibitem{mu-19}
Merker, L., Ueckerdt, T.: Local and union page numbers. In: Archambault, D.,
  T{\'o}th, C.D. (eds.) Graph Drawing and Network Visualization. pp. 447--459.
  Springer International Publishing, Cham (2019).
  \doi{10.1007/978-3-030-35802-0\_34}

\bibitem{mu-20}
Merker, L., Ueckerdt, T.: The local queue number of graphs with bounded
  treewidth. In: Auber, D., Valtr, P. (eds.) Graph Drawing and Network
  Visualization. pp. 26--39. Springer International Publishing, Cham (2020).
  \doi{10.1007/978-3-030-68766-3\_3}

\bibitem{mww-88}
Muder, D.J., Weaver, M.L., West, D.B.: Pagenumber of complete bipartite graphs.
  Journal of Graph Theory  \textbf{12}(4),  469--489 (1988).
  \doi{10.1002/jgt.3190120403}

\bibitem{n-64}
Nash-Williams, C.S.: Decomposition of finite graphs into forests. Journal of
  the London Mathematical Society  \textbf{s1-39}(1),  12--12 (1964).
  \doi{10.1112/jlms/s1-39.1.12}

\end{thebibliography}

\end{document}